%% file: Batyrev-Seidel-selecta.tex
\documentclass{amsart}

\usepackage{pinlabel}
\usepackage{amsmath,amsfonts,amssymb,latexsym,mathrsfs,stmaryrd}
\usepackage{color}
\usepackage{graphicx}
\usepackage{graphics}
\usepackage{enumerate}
\usepackage{amscd} 
\usepackage{amsxtra}
\usepackage{epic,eepic}
\usepackage{hyperref} 

\newtheorem{theorem}{Theorem}[section]
\newtheorem{prop}[theorem]{Proposition}
\newtheorem{proposition}[theorem]{Proposition}
\newtheorem{corollary}[theorem]{Corollary}

\newtheorem{lemma}[theorem]{Lemma}

\theoremstyle{definition} 

\newtheorem{definition}[theorem]{Definition}

\newtheorem{remark}[theorem]{Remark}

\newtheorem{notation}[theorem]{Notation} 
\newcommand{\cA}{\mathcal{A}}

\newcommand{\cD}{\mathcal{D}}

\newcommand{\cO}{\mathcal{O}}
\newcommand{\cU}{\mathcal{U}}

\newcommand{\ol}[1]{\overline{#1}}

\newcommand{\ran}{\right\rangle}
\newcommand{\lan}{\left\langle}

\newcommand{\lie}[1]{\ensuremath{\mathfrak{#1}}}

\renewcommand{\t}{\lie{t}}

\newcommand{\tensor}{\otimes}

\newcommand{\iso}{\cong}

\renewcommand{\P}{\bb{P}} %wird p appears sometimes so use rename
\newcommand{\C}{\bb{C}}

\renewcommand{\L}{\bb{L}}%\L is L with a line through
\newcommand{\Z}{\bb{Z}}
\newcommand{\R}{\bb{R}}
\newcommand{\Q}{\bb{Q}}
\newcommand{\T}{\bb{T}}

\newcommand{\qu}{/\kern-.7ex/}
\newcommand{\lqu}{\backslash \kern-.7ex \backslash}%left
\newcommand{\on}{\operatorname} 
\newcommand{\Aut}{\on{Aut}}
\newcommand{\Ham}{\on{Ham}} 
\newcommand{\Hom}{\on{Hom}}

\newcommand{\Pic}{\on{Pic}}

\newcommand{\NE}{\on{NE}}
\newcommand{\PD}{\on{PD}} 

\newcommand{\bb}[1]{\ensuremath{\mathbb{#1}}}

\newcommand{\tS}{{\widetilde S}}
\newcommand{\tD}{{\widetilde D}}
\newcommand{\tp}{\tilde{p}} 

\newcommand{\hD}{\widehat{D}} 
\newcommand{\hphi}{\hat{\phi}}
\newcommand{\hb}{\hat{b}} 
\newcommand{\hE}{\widehat{E}} 

\newcommand{\F}{\mathbb{F}} 
\newcommand{\id}{\operatorname{id}}

\newcommand{\ov}{\overline} 

\title{Seidel elements and mirror transformations}
\author{Eduardo Gonz\'alez }
\address{Department of Mathematics\\
UMASS Boston\\100 Morrisey Blvd\\
Boston, MA 02125. USA.}
\email{eduardo@math.umb.edu}
\urladdr{http://www.math.umb.edu/~eduardo/}
\thanks{E.G. is supported by NSF grant DMS-1104670 and H.I. is
  supported by Grant-in-Aid for Young Scientists (B) 22740042 }

\author{Hiroshi Iritani} 
\address{Department of Mathematics, Graduate School of Science\\
  Kyoto University\\
  Oiwake-cho, Kitashirakawa, Sakyo-ku, Kyoto, 606-8502, Japan.}
\email{iritani@math.kyoto-u.ac.jp}

\begin{document}
\begin{abstract} 
  The goal of this article is to give a precise relation between the
  mirror symmetry transformation of Givental and the Seidel elements
  for a smooth projective toric variety $X$ with 
  $-K_X$ nef. 
  We show that the Seidel elements entirely determine 
  the mirror transformation and mirror coordinates. 
\end{abstract}
\maketitle 

\begin{center}
\today
\end{center}
\tableofcontents

\section{Introduction}

Let $X$ be a smooth projective toric variety. The variety $X$
can be explicitly written as the symplectic reduction of the Hermitian
space $\C^m$ by a Hamiltonian action of a torus 
$(S^1)^r$, where $r$ is the Picard number of $X$. 
Let $D_1,\dots, D_m$ denote the classes in
$H^2(X)$ Poincar\'e dual to the toric divisors. Let $t_i$ denote the
coordinates in $H^2(X)$ with respect to an integral, nef basis
$p_1,\dots, p_r$, and let $q_i=\exp(t_i)$ be the exponential
coordinates. 
Recall that the mirror theorem of Givental \cite{Gi-A-98}
states that if $c_1(X)=-K_X=D_1+\dots + D_m$ 
is semipositive (nef),
then the cohomology valued function (of the B-model)
\[
I_X(y,z)=e^{\sum_{i=1}^r p_i\log y_i/z} 
\sum_{d\in \NE(X)_\Z} \prod_{i=1}^m 
\left(
\frac{\prod_{k=-\infty}^0
  (D_i+kz)}{ \prod_{k=-\infty}^{\lan D_i,
    d\ran}(D_i+kz)}
\right)
y_1^{d_1}\dots y_r^{d_r}
\]
determines the J-function $J_X(q,z)$ 
(of the the $A$-model or Gromov-Witten theory) 
\begin{equation} 
\label{eq:JX} 
J_X(q,z) = e^{\sum_{i=1}^r p_i \log q_i/z} 
\left( 1 + 
\sum_{\alpha} 
\sum_{d\in \NE(X)_\Z \setminus\{0\}} 
\lan \frac{\phi_\alpha}{z(z-\psi)} \ran^X_{0,1,d} 
\phi^\alpha q_1^{d_1} \cdots q_r^{d_r}\right) 
\end{equation} 
via a change of coordinates 
$\log q_i= \log y_i + g_i(y)$, $i=1,\dots,r$, 
in such a way that $I_X(y,z)=J_X(q,z)$. 
Here the variables $y_1,\dots,y_r$ of the B-model 
are called \emph{mirror coordinates} and this 
change of variables is called \emph{mirror transformation}  
(or \emph{mirror map}). 
This relation can be used to show that 
the small quantum cohomology ring $QH(X)$ differs
from the original presentation suggested by Batyrev \cite{Ba-Qu93}
only by this change of coordinates. 
We refer to Givental \cite{Gi-A-98} and 
the text by Cox and Katz \cite{CoKa-Mi99} for
further details on this discussion.

Let $(Y,\omega)$ denote a symplectic manifold. 
For a loop $\lambda$ in the group of Hamiltonian 
symplectomorphisms on $Y$, Seidel 
\cite{Se-pi97}, assuming $Y$ monotone, 
constructed an invertible element $S(\lambda)$ 
in quantum cohomology counting sections of the 
associated clutched Hamiltonian fibration 
$E_\lambda \to \P^1$ with fibre $Y$. 
The \emph{Seidel element} $S(\lambda)$ defines an element
in $\Aut(QH(Y))$ via quantum multiplication, and the
association $\lambda \mapsto S(\lambda)$ a representation of
$\pi_1(\Ham(Y))$ on $QH(Y)$. This construction was later extended
for all symplectic manifolds, see for instance McDuff and Tolman
\cite{McTo-To06} where Seidel's construction was used to study the
underlying symplectic topology in toric manifolds. 

In the case where the loop $\lambda$ 
is a (relatively simple) circle action, 
the asymptotic form 
of $S(\lambda)$ can be written explicitly in terms of
geometric and Morse theoretic information 
\cite[Theorem 1.10]{McTo-To06}. 
Regarding $X$ as a Hamiltonian space, they consider
the Seidel element\footnote{Here $\tS_j$ is a variant 
of the Seidel element $S_j = S(\lambda_j)$ given by 
$S_j = q_0 \tS_j$, where $q_0$ is the variable 
corresponding to the maximal section class of the 
associated bundle $E_j$.}
$\tS_j$ associated to an action $\lambda_j$ that
fixes the toric divisor $D_j$ 
(see Section \ref{subsec:action_fixing_divisor}). 
It is proved that $\tS_j$ is a series of
the form $\tS_j = D_j+ O(q)$
if $-K_X$ is nef, and $\tS_j = D_j$ 
in the Fano case ($-K_X$ is ample). 
Moreover, it is shown that these elements 
satisfy the following Batyrev's relations: 
\begin{equation} 
\label{eq:Batyrev_relation}
\prod_{j:\lan D_j, d\ran>0} 
\tS_j^{\lan D_j, d\ran} 
= q^d \prod_{j:\lan D_j, d \ran <0} \tS_j^{-\lan D_j, d\ran} 
\quad \text{ in $QH(X)$}   
\end{equation} 
where $d\in H_2(X,\Z)$. 
This shows that $QH(X)$ is abstractly isomorphic to 
the Batyrev presentation \cite[Proposition 5.2]{McTo-To06}. 
Based on this, one can conjecture that Seidel
elements should be closely related 
or even equivalent to the mirror transformation. 
In a recent paper Fukaya, Oh, Ohta and Ono
\cite{FuOh-La10A} have used Seidel elements 
in the mirror symmetry context as well. 

We begin with a calculation of Seidel elements. 
It turns out that $\tS_j$ appears 
as a coefficient of the J-function $J_{E_j}$ 
of the associated bundle $E_j$ 
(Proposition \ref{p:SeidelfromJ}), 
and thus one can use the mirror transformation for $E_j$ 
(which is toric as well) to calculate $\tS_j$. 
Let $(m_{ij})$ be the matrix of toric divisors 
such that $D_j = \sum_{i=1}^r m_{ij} p_i$.  
As mirror analogues of Seidel elements, 
we introduce the \emph{Batyrev elements} 
$\tD_j = D_j +O(y) \in H^*(X)\llbracket y \rrbracket$ 
by 
\begin{equation} 
\label{eq:Batyrev_element} 
\tD_j = \sum_{i=1}^r m_{ij} \tp_i \quad \text{with} \quad 
\tp_i = \sum_{k=1}^r 
\frac{\partial \log q_k}{\partial \log y_i} p_k.  
\end{equation} 
These elements satisfy Batyrev's relations \eqref{eq:Batyrev_relation} 
with $q^d$ there replaced by $y^d$ (Proposition 
\ref{p:bat-rel}). We can calculate 
them as explicit hypergeometric series in $y$ 
(Lemma \ref{l:tDj}). 
Note that the Batyrev elements and the Seidel 
elements satisfy the \emph{same} relation in 
\emph{different} coordinates. 
We see that they only differ by a function multiple. 

\begin{theorem}[Theorem \ref{th:batsei}, Lemma \ref{l:g0}]
Let $X$ be a smooth projective toric variety 
with $-K_X$ nef. 
The Seidel element associated to the toric divisor $D_j$ is given by
\[
\tS_j(q_1\dots q_r)= 
\exp\left(-g_0^{(j)} (y_1,\dots,y_r)\right) 
\tD_j(y_1,\dots y_r).  
\] 
The correction term $g_0^{(j)}$ here 
is an explicit hypergeometric series \eqref{eq:g0} in $y$.   
\end{theorem}

Next we ask the converse: whether the Seidel 
elements determine the mirror transformation. 
One can see from the definition \eqref{eq:Batyrev_element} 
that the Batyrev elements $\tD_1,\dots,\tD_m$ can 
determine the Jacobian of the mirror map, in particular, 
the mirror map itself. 
Therefore, it is important to know when $g_0^{(j)}$ vanishes. 
We show that $g_0^{(j)}$ vanishes if and only if 
$-K_X$ is \textbf{big} on the toric divisor $D_j$, i.e.\ 
$(-K_X)^{n-1}\cdot D_j>0$.
In terms of 
the fan of $X$, this is also equivalent to 
the primitive generator $b_j$ 
of the corresponding 1-cone being 
a vertex of the fan polytope (Proposition \ref{p:vanishing}). 
Another important fact is that 
the Batyrev elements $\tD_j$ satisfy the 
same linear relations as the toric divisors do: 
\begin{equation} 
\label{eq:linear_relation} 
\sum_{i=1}^m c_j \tD_j =0 \quad 
\text{whenever} \quad 
\sum_{i=1}^m c_j D_j = 0.  
\end{equation} 
The Seidel elements $\tS_j$ do not necessarily 
satisfy the linear relations. 
The partial vanishing of $g_0^{(j)}$ and these  
linear relations are enough to reconstruct the Batyrev 
elements from the Seidel elements.

\begin{theorem}%[Theorem \ref{t:bats}]
\label{t:intro3} 
For a smooth projective toric variety $X$ 
with $-K_X$ nef, the Seidel elements $\tS_j$ 
entirely determine the Batyrev elements $\tD_j$ 
and in particular the mirror transformation. 
More precisely the Batyrev elements 
$\tD_j \in H^*(X)\llbracket q \rrbracket$, 
$j=1,\dots,m$ are uniquely characterized by  
the following conditions: 
\begin{enumerate}
\item $\tD_j = H_j \tS_j$ for some $H_j\in \Q\llbracket q\rrbracket$; 
\item $\tD_j = \tS_j$ if $(-K_X)^{n-1}\cdot D_j>0$; 
\item $\tD_j$ satisfy the linear relations 
\eqref{eq:linear_relation}. 
\end{enumerate} 
\end{theorem}

We add a remark on mirror symmetry. 
The mirror coordinates $y_1,\dots,y_r$ here 
represent the complex moduli of 
the mirror Landau-Ginzburg model. 
There are no preferred choices of coordinates on 
the complex moduli space and both $q$ and $y$ 
serve as local coordinates. 
However the $y$ coordinates are considered to 
represent a global algebraic structure of the 
complex moduli space. 
Therefore our result suggests that Gromov-Witten 
theory itself (via torus action and Seidel elements) 
can reconstruct a global algebraization of 
the K\"{a}hler moduli space which is a priori 
a formal germ at $q=0$.  

% Our theorem implies that, in the case of toric varieties, 
% the Seidel elements i.e.\ the Gromov-Witten invariants 
% can determine global algebraic coordinates $y$ of 
% the K\"{a}hler moduli space 
% which is mirror to the complex moduli space 
% and give rise to an algebraization of it. 

% The proof of this theorem relies on the observation that although
% Seidel elements satisfy the multiplicative Batyrev relations they do
% \emph{not} satisfy the linear Batyrev relations. This allow to
% reconstruct Batyrev elements from Seidel elements. 

\textbf{Acknowledgments.} 
We thank David Cox, Dusa McDuff, Kaoru Ono
and Chris Woodward for useful discussions regarding an early
draft of the paper. We also thank Changzheng Li and Kwokwai Chan
for helpful comments on an earlier version of the paper.  The authors are grateful for the hospitality at
the Centre International de Rencontres Math\'{e}matiques in Luminy, 
where this project was started. E.G. is supported by NSF grant DMS-1104670 and H.I. is
  supported by Grant-in-Aid for Young Scientists (B) 22740042.

\section{Seidel elements and J-functions}\label{s:SJ}
We introduce the main notation and constructions of Seidel
element and explain its relation to the J-function.

\subsection{Generalities} 
We begin with the definition of the Seidel element. 
Let $X$ be a smooth projective variety, equipped with a
$\C^\times$ action. 
Each $\C^\times$-orbit in $X$ 
contains a fixed point in its closure, 
and thus the associated $S^1$-action on $X$ 
is Hamiltonian by Frankel's theorem \cite{Fr-59}. 
In this paper we will restrict to this case,
however the construction works more generally 
in the symplectic category. 
The original definition is due to Seidel \cite{Se-pi97} for
monotone symplectic manifolds. 
The reader can consult \cite{McSa-J04} 
for a detailed exposition and \cite{LaMcPo-To99,Mc-Qu00} 
for the construction in general case. 
For relations with the computation of small
quantum cohomology rings see 
\cite{Mc-Qu00, Go-Qu06, Mc-Sy06, McTo-To06}.

\begin{definition}\label{eq:a-bundle}
The \textbf{associated bundle} of the $\C^\times$-action 
is the $X$-bundle over $\P^1$
\begin{equation*}
% \label{eq:bundle}
    E:=X\times (\C^{2}\setminus \{0\}) / \C^{\times}\to \P^1,
\end{equation*}
  where $\C^\times$ acts with the standard diagonal action
  $\lambda\cdot (x, z)=(\lambda x, \lambda z)$.
\end{definition}

Let $\phi_1,\dots, \phi_N$ denote a basis for the rational
cohomology $H^*(X;\Q)$. By abuse of notation, every time we omit
coefficients we mean rational cohomology. Let $\phi^1,\dots, \phi^N$
denote the dual basis, that is $(\phi_i,\phi^j)=\delta_{ij}$ where
$(\cdot,\cdot)$ is the usual pairing in cohomology. 
There is a (non-canonical) splitting \cite{Mc-Qu00} 
\begin{equation*}
%%% \label{eq:split}
  H^*(E) \cong H^*(X)\tensor H^*(\P^1).
\end{equation*}
We let $\hphi_1,\dots, \hphi_M$ denote a basis for $H^*(E)$, 
and let $\hphi^1,\dots, \hphi^M$ denote the dual basis. 
There is a unique $\C^\times$-fixed component $F_{\rm max} 
\subset X^{\C^\times}$ such that the 
normal bundle of $F_{\rm max}$ has only negative 
$\C^\times$-weights. 
When we take a Hamiltonian function 
for the $S^1$-action, 
$F_{\rm max}$ is the maximum 
set of the Hamiltonian. 
Each fixed point $x\in X^{\C^\times}$ 
defines a section $\sigma_x$ of $E$. 
We denote by $\sigma_0$ the section associated 
to a fixed point in $F_{\rm max}$. 
This maximal section defines a splitting 
\begin{equation} 
\label{eq:split_maximal}
H_2(E,\Z)/\text{tors} \cong 
\Z [\sigma_0] \oplus 
(H_2(X,\Z)/\text{tors}) 
\end{equation} 
Let $\NE(X)\subset H_2(X,\R)$ denote the Mori cone, 
that is the cone generated by effective curves 
and set $\NE(X)_\Z := \NE(X) \cap (H_2(X,\Z)/\text{tors})$. 
We introduce $\NE(E)$ and $\NE(E)_\Z$ similarly. 
\begin{lemma} \label{l:split}
$\NE(E)_\Z = \Z_{\ge 0} [\sigma_0] + \NE(X)_\Z$. 
\end{lemma} 
\begin{proof} 
The associated bundle $E$ has the 
$T^2 = \C^\times \times \C^\times$ action 
defined by $(t_1,t_2) \cdot [x,(z_1,z_2)] 
=[t_1 x, (z_1, t_2 z_2)]$. 
By the $T^2$-action, every curve 
can be deformed to a sum of $T^2$-invariant curves 
in the same homology class. 
A $T^2$-invariant curve is either contained 
in the fibres at $0$, $\infty$ or 
in the subspace $F\times \P^1 \subset E$ 
for some fixed component $F\subset X^{\C^\times}$. 
Therefore it suffices to show that the 
section class $[\sigma_x]$ associated to 
a fixed point $x\in X^{\C^\times}$ 
is of the form 
$[\sigma_0]+ d$ for some $d\in 
\NE(X)_\Z$. 
Take a nontrivial $\C^\times$-orbit $O$ in $X$ 
and consider its closure $\ov{O} \cong \P^1$. 
This gives an embedding of the 
Hirzebruch surface 
$\F_k \cong \ov{O}\times (\C^2\setminus \{0\})/\C^\times$ 
into $E$, where $k$ is the order of the 
stabilizer at $O$. It is well known that the 
maximal section class in $\F_k$ is the sum 
of the minimal section class and $k$ times 
the fibre class. By joining $x$ with the maximal 
fixed component $F_{\rm max}$ by a chain of 
$\C^\times$-orbits, we obtain a relation 
$[\sigma_x]= [\sigma_0] + d$ with $d$ a sum 
of $\C^\times$-invariant curves in $X$.  
\end{proof} 

Let $H_2^{\sec}(E,\Z)$ denote the affine 
subspace of $H_2(E,\Z)/\text{tors}$ 
which consists of section classes, 
that is the classes that project to 
the positive generator of $H_2(\P^1,\Z)$. 
We set $\NE(E)^{\sec}_\Z := 
\NE(E)_\Z \cap H_2^{\sec}(E,\Z)$. 
The above lemma shows that 
\begin{equation}
\label{eq:split2}
\NE(E)^{\sec}_\Z=  [\sigma_0] + \NE(X)_\Z.  
\end{equation} 
Let $r$ be the rank of $H^2(X)$. 
We choose an integral basis $\{p_1,\dots,p_r\}$ of $H^2(X)$ 
which pairs non-negatively with $\NE(X)_\Z$ (i.e.\ $p_i$ is nef). 
There are unique lifts of $p_1,\dots,p_r$ in $H^2(E)$ 
which vanish on $[\sigma_0]$. 
We denote these lifts by the same symbols 
$p_1,\dots,p_r$. By the above lemma, 
these lifts are also nef. 
Let $p_0\in H^2(E)$ denote the pull-back 
of the positive generator of $H^2(\P^1,\Z)$. 
Then $\{p_0,p_1,\dots,p_r\}$ forms an  
integral basis of $H^2(E)$. 
Let $q_1,\dots,q_r$ denote the Novikov variables of $X$ 
dual to the basis $p_1,\dots,p_r$. 
Similarly let $q_0,q_1,\dots,q_r$ denote the Novikov 
variables of $E$ dual to $p_0,p_1,\dots,p_r$. 
We take 
\[ 
\Lambda_X :=\Q\llbracket q_1,\dots, q_r \rrbracket,  
\qquad 
\Lambda_E := \Q\llbracket q_0,q_1,\dots,q_r\rrbracket 
\] 
to be the Novikov rings of $X$ and $E$ respectively. 
We write 
\begin{align*} 
q^d &:= q_1^{\lan p_1, d\ran}  
\cdots q_r^{\lan p_r, d \ran} \in \Lambda_X 
\quad 
\text{for } d\in \NE(X)_\Z   \\ 
q^{\beta} & := 
q_0^{\lan p_0, \beta \ran} q_1^{\lan p_1, \beta \ran} 
\cdots q_r^{\lan p_r, \beta \ran} 
\in \Lambda_E  
\quad 
\text{for } \beta \in \NE(E)_\Z. 
\end{align*} 
The small quantum cohomology ring 
\[
QH(X) = \left(
H(X) \otimes_\Q \Lambda_X, \bullet 
\right) 
\]
is defined over the Novikov ring $\Lambda_X$. 
Let $\lan\cdots \ran^{X}_{g,k,d}$
(resp. $\lan\cdots \ran^{E}_{g,k,d}$) 
denote the genus $g$, degree $d$
Gromov-Witten invariant of $X$ (resp.\ $E$) 
with $k$ insertions. We refer the reader to 
\cite{CoKa-Mi99} and references therein 
for the definition of algebraic Gromov-Witten invariants. 
Since the proof of Givental's mirror theorem \cite{Gi-Eq96, Gi-A-98} 
is based on algebraic geometry, we will work with 
algebraic Gromov-Witten invariants.

\begin{definition}\label{def:seidel}
  The \textbf{Seidel element} of $X$ is
  the class
  \begin{equation}\label{eq:seidel}
    S:=\sum_\alpha\sum_{\beta \in \NE(E)_\Z^{\sec}} 
    \lan
    \iota_*\phi_\alpha
    \ran^{E}_{0,1,\beta} \ \phi^{\alpha} q^{\beta}
  \end{equation}
in $QH(X)\tensor_{\Lambda_X}\Lambda_E$. 
Here $\iota\colon X \to E$ denotes the inclusion of a fibre. 
By Equation \eqref{eq:split2}, the Seidel element 
can be factorized as $S = q_0 \tS$ with $\tS \in QH(X)$. 
\end{definition}

In general, one can define the Seidel element 
$S(\lambda)$ for a loop $\lambda$ 
in the group $\Ham(X)$ of Hamiltonian diffeomorphisms and 
one gets a representation of $\pi_1(\Ham(X))$ 
on $QH(X)$ via the quantum multiplication by $S(\lambda)$. 
In our simple situation, this fact can be stated as follows. 
Suppose we have two commuting $\C^\times$-actions 
$\lambda_1,\lambda_2$. Let $\lambda_3=\lambda_1*\lambda_2$ 
be the composite $\C^\times$-action. 
Let $E_i$, $S_i$, $i=1,2,3$ be the $X$-bundle and 
the Seidel element associated to $\lambda_i$. 
The two commuting $\C^\times$-actions define 
the associated $X$-bundle $\hE$ over $\P^1\times\P^1$ 
such that the restriction to $\P^1\times \{z\}$ (resp.\ 
$\{z\} \times \P^1$) is isomorphic to the bundle 
$E_1$ (resp.\ $E_2$). 
Then $E_3$ can be obtained as the restriction 
of $\hE$ to the diagonal in $\P^1\times\P^1$. 
From this geometry we have a natural map\footnote
{In symplectic topology, $E_3$ is isomorphic to
the fibre sum $E_1\# E_2$.} 
\[
H_2^{\sec}(E_1,\Z) \times H_2^{\sec}(E_2,\Z) 
\to H_2^{\sec}(E_3,\Z).  
\] 
Under the multiplication of Novikov variables induced by this 
map, we have 
\[
S_3 = S_1 \bullet S_2. 
\]
By considering the inverse $\C^\times$-action, 
we find that the Seidel element is invertible 
in $QH(X)$ since the trivial $\C^\times$-action 
gives rise to the trivial Seidel element $q_0 1$. 
% \begin{theorem}[Seidel]\label{t:inv-sei}
% $S$ is an invertible element in $(QH(X),\bullet)$. 
% \end{theorem}

\subsection{J-functions}
% Let $q=(q_0,q_1,\dots, q_r)$ denote the exponential coordinates in
% $H^2(E)$ as before, where $q_0$ is the section class and
% $q_i,i=1,\dots r$ correspond to the coordinates in $\iota_* H^2(X)$.
% Using the abbreviated notation
% \[
% p\log q :=\sum_{i=0}^r p_i\log q_i,\ \  p\log q/z
% :=\frac{1}{z}\left(\sum_{i=0}^r p_i\log q_i\right), 
% \] 

% we have.
\begin{definition}[\cite{Gi-Eq96, Gi-A-98}]
\label{d:JE} 
The (small) \textbf{J-function} of $E$ is the
cohomology valued function
\[
J_E(q,z)=e^{\sum_{i=0}^r p_i\log q_i/z} \left( 
1+ \sum_{\alpha=1}^M \sum_{
\beta \in \NE(E)_\Z \setminus \{0\}} 
\lan \frac{\hphi_\alpha}{z(z-\psi)}
\ran^E_{0,1,\beta}\hphi^{\alpha} q^\beta
\right),
\]
where $z$ is a formal variable, and $\psi$ is the first Chern class of
the universal cotangent line bundle $\mathcal{L}\to
\ol{\mathcal{M}}_{0,1}(E,d)$ at the marked point. 
The fraction $\hphi_\alpha /(z(z-\psi))$ in the 
correlator should be expanded in the series 
$\sum_{n\ge 0} z^{-2-n} \hphi_\alpha \psi^n$. 
The J-function $J_X(q,z)$ of $X$ 
is defined similarly (see Equation \eqref{eq:JX}). 
\end{definition}

In order to see the relation with the Seidel elements, 
we expand $J_E$ in terms of powers of $z$ as follows.

\begin{equation*}
%  \label{eq:JE}
  J_E(q,z)=e^{\sum_{i=0}^r p_i\log q_i/z} 
\left( 
    1+ z^{-2} \sum_{n=0}^{\infty} F_n(q_1,\dots,q_r) q_0^n +
    O(z^{-3}) 
\right),
\end{equation*}
where the functions $F_n(q_1,\dots,q_r)$ are power series with values
in $H^*(E)$.

\begin{proposition}
\label{p:SeidelfromJ}
The Seidel element of the action is 
given by $S= \iota^* (F_1(q_1,\dots,q_r) q_0)$.
\end{proposition}

\begin{proof}
% From Definition \ref{d:JE} 
% we get
% 
%\begin{equation*}
%  \label{eq:2}
%  J_E(q,z)=
%  e^{\sum_{i=0}^r p_i\log q_i/z} 
%  \left( 
%    1+ z^{-2}\sum_{d\in \NE(X)_\Z}\sum_{n=0}^{\infty}\lan
%    \hphi_{\alpha}\ran^E_{0,1,d+n\sigma_0} q_0^n  \ q_1^{d_1}\cdots
%    q_r^{d_r}\hphi^{\alpha} + 
%    O(z^{-3}) 
%  \right). 
%\end{equation*} 
From Definition \ref{d:JE} we find 
\begin{equation*}
%\label{eq:Fn}
F_n(q_1,\dots,q_r)=\sum_{\alpha=1}^M 
\sum_{d\in \NE(X)_\Z}\lan
    \hphi_{\alpha}\ran^E_{0,1,d+n\sigma_0} \ q^d 
\hphi^{\alpha}.
\end{equation*}
Using the duality identity 
\begin{equation*}
%\label{eq:duality}
\sum_{\alpha=1}^M \hphi_\alpha \tensor \iota^* \hphi^\alpha =
\sum_{\alpha=1}^N \iota_*\phi_\alpha\tensor\phi^\alpha, 
\end{equation*} 
we get 
\begin{equation}
\label{eq:pull}
\iota^* F_1(q_1,\dots,q_r)
=\sum_{\alpha=1}^N \sum_{d\in \NE(X)_\Z}
\lan i_*\phi_{\alpha}\ran^E_{0,1,d+\sigma_0} 
\ q^d \phi^{\alpha}.
\end{equation}
The conclusion follows from 
Equations \eqref{eq:seidel} and 
\eqref{eq:pull}.
\end{proof}

\section{Seidel elements for toric manifolds}

%%%%
\subsection{Notation}\label{ss:not}

We now fix some notation on toric geometry for this paper. 
For more details see \cite{Au-To04, CoKa-Mi99, CoLi-To10}. 
For this paper a \textbf{toric manifold} $X$ is a
projective smooth toric variety, as constructed from the
following data.

\begin{enumerate}
\item An integral lattice $M \cong \Z^n$ 
and its dual $N=\Hom(M,\Z)$. 
We denote by $\lan \cdot ,\cdot\ran$ the natural pairing 
between $N$ and $M$. 
\item A fan $\Sigma$ in $N_\R:=N\tensor \R$ consisting of a collection
  of strongly convex rational polyhedral cones $\sigma\subset N_\R$,
  which is closed under intersections and taking faces.
\end{enumerate}

We shall assume that the fan $\Sigma$ is 
complete and regular.
Let $\Sigma(1)$ denote the set of $1$-cones (rays) in $\Sigma$, 
and we let $b_i$ denote the set of integral primitive generators 
of the $1$-cones. 
The group $N$ is the lattice of the torus $N\tensor \C^\vee$ 
and thus $M$ is the
lattice of characters in $N\tensor \C^\vee$.  
The \emph{fan sequence} of $X$ is the exact sequence 
\begin{equation}
\label{eq:fans}
\begin{CD} 
0 @>>> \L @>>> \Z^m @>>> N @>>> 0,
\end{CD}
\end{equation}
where the second map takes the canonical basis to the primitive
generators $b_1,\dots, b_m$ 
and $\L$ is defined to be the kernel of the second map. 
This in turn defines a torus $\T=\L\tensor
\C^\times$, with character and weight lattices $\L,\L^\vee$
respectively and a sequence
\begin{equation*}
%\label{eq:fan-t}
\begin{CD} 
  0 @>>> \T @>>> (\C^\times)^m@>>> 
N\tensor \C^\times @>>> 0.
\end{CD} 
\end{equation*}
The dual of the sequence \eqref{eq:fans} is the \emph{divisor sequence}
\begin{equation}
\label{eq:divs}
\begin{CD}
  0 @>>>  M @>>> 
(\Z^m)^\vee @>>>  \L^\vee @>>> 0.
\end{CD} 
\end{equation}
The first arrow takes $v\in M$ into the tuple  $(\lan
v,b_i\ran)^m_{i=1}$. The images of the canonical basis under the
second map will be denoted by $D_i$, $i=1,\dots,m$. 

The weights $D_i$ give an homomorphism $\T\to (\C^\times)^m$, and we let
the torus $\T$ act on $\C^m$ via this homomorphism. The combinatorics
of the fan defines a stability condition of this action as
follows. Let $Z(\Sigma)$ denote the union
\begin{equation} 
\label{eq:del}
Z(\Sigma):=\bigcup_{I\in \cA} \C^I, 
\quad \C^I=\{(z_1,\dots,z_m):z_i=0 \text{ for }
i\notin I\}.
\end{equation} 
where $\cA$ is the collection of anti-cones, that is the subsets of
indices that do not yield a cone in the fan 
\[\cA:=\left\{I \ : \ \sum_{i\in I}
\R_{\geq 0}b_i \notin \Sigma\right\}.\]
The toric variety $X$ is defined as the quotient
\begin{equation*}
%\label{eq:quot}
X:=[\cU/\T];\ \quad 
\cU:=\C^m \setminus Z(\Sigma).
\end{equation*}

Each character $\xi:\T\to \C^\times$ defines a line bundle 
\[
L_{\xi}:=\C\times_{\xi,\T} \cU\to X. \]
The correspondence $\xi\mapsto L_\xi$ yields an identification of the
Picard group with the character group of $\T$. 
Thus, we have
\[
 \L^\vee=\Hom(\T,\C^\times) \iso \Pic(X) 
\stackrel{c_1}{\iso} H^2(X,\Z). 
\]
The Poincar\'e dual of the prime toric divisor 
$\{z_i=0\}\subset X$ is the image of $D_i$ in $H^2(X,\Z)$. 
By abuse of notation, $D_i$ denotes both 
the divisor $\{z_i=0\}$ itself and its class in 
$H^2(X,\Z) \cong \L^\vee$. 
We note that $\L=H_2(X,\Z)$.

The K\"ahler cone $C_X$ of $X$, 
the cone consisting of K\"ahler classes,
is given by 
\[
C_X:=\bigcap_{I\in \cA} \sum_{i\in I} \R_{>0} D_i \subset
\L^\vee\tensor \R=H^2(X;\R).
\]

We assume that $C_X$ is nonempty so that $X$ is projective. 
We will need later the following notation. 
As before $p_1,\dots, p_r\in H^2(X,\Z)$ denote
a nef integral basis, that is an integral basis such that
\(
p_a\in \ov{C}_X.
\)
Then we write the toric divisors as
\begin{equation}
\label{eq:div-basis}
D_j=\sum_{i=1}^r m_{ij}p_i, 
\end{equation}
for some $m_{ij}$.  
The Mori cone $\NE(X)\subset H_2(X,\R)$ 
is the dual of the cone $\ov{C}_X$. 
As before $\NE(X)_\Z$ denotes 
the semi-group $\NE(X)\cap H_2(X,\Z)$.

We shall now explain the symplectic structure of $X$. 
Take $\T_\R$ to be maximal compact in $\T$. 
%and let $\T_\R$ act on $\C^m$ via
%
%\[
%\exp(\xi) (z_1,\dots, z_m)= (e^{\lan D_1, \xi \ran} z_1, \dots
%,e^{\lan D_m, \xi \ran} z_m), \ \ \xi \in \t=\Lie(\T_\R),
%\]
%
The $\T_\R$-action on $\C^m$ is generated by 
the Hamiltonian 
\[ 
h\colon \C^m \to \t^\vee_\R, \ \ \
h(z_1,\dots, z_m)=\sum_{i=1}^m|z_i|^2D_i.
\]  
Taking a K\"ahler class
$\eta\in C_X$, we have an homeomorphism (cf. \cite{Au-To04, Gu-Mo94})
\[
h^{-1}(\eta)/T_\R\iso X,
\]
which induces a symplectic structure (still denoted by) $\eta$ on $X$.  This fact
and the equivalence of the algebraic and symplectic Gromov-Witten
invariants \cite{LiTi-Co99} yield the following expected result.

\begin{lemma}
  The Seidel element of the symplectic toric manifold
  $(X,\eta)$ as defined in \cite{Se-pi97, McTo-To06} coincides
  with the one in Equation \eqref{eq:seidel}.
\end{lemma}

\subsection{The $\C^\times$-action fixing a toric divisor}
\label{subsec:action_fixing_divisor}
For each divisor $D_j$ we take a $\C^\times$-action on
$X$ rotating around $D_j$ 
and describe the geometry of its associated bundle.  

Consider the action of $\C^\times$ on $\C^m$ given by
\[
(z_1,\dots,z_m)\mapsto 
(z_1,\dots,t^{-1}z_j, \dots, z_m), \ \ t\in
\C^\times 
\]
and the induced action on 
$X = (\C^m\setminus Z(\Sigma)) / \T$. 
The toric divisor $D_j= \{z_j=0\}$ is the 
maximal fixed component of this action. 
We extend this to the diagonal 
$\C^\times$-action on  
$(\C^m\setminus Z(\Sigma)) 
\times (\C^2\setminus \{0\})$ by
\[
(z_1,\dots,z_m, u,v)\mapsto (z_1,\dots,t^{-1}z_j, \dots, z_m, tu,tv),
\ \ t\in \C^\times.
\]
The associated bundle $E_j$ of 
the $\C^\times$-action on $X$ is given by 
\begin{equation*}
%\label{eq:bundlej}
E_j=(\C^m\setminus Z(\Sigma))\times
(\C^2\backslash \{0\}) / \T\times \C^\times. 
\end{equation*}
Therefore $E_j$ is also a toric variety. 
We can identify $H^2(E_j,\Z)$ with the 
lattice of characters of $\T\times \C^\times$: 
\begin{equation}
\label{eq:splitE}
H^2(E_j,\Z) \cong \L^\vee\oplus \Z 
\cong H^2(X,\Z)\oplus \Z.
\end{equation}
This is dual to the splitting 
in Equation \eqref{eq:split_maximal}. 
In light of this splitting, the $m+2$ weights of 
$\T\times \C^\times$ defining $E_j$ 
are just given by
\begin{equation}
\label{eq:divE}
\hD_i=(D_i,0) \ \text{ for }\ i\neq j; 
\quad \hD_j=(D_j,-1);\quad 
\hD_{m+1}=\hD_{m+2}=(\vec{0},1).
\end{equation}
This in turn yields the divisor sequence
\[
\begin{CD}
0 @>>> M\oplus\Z @>>> 
(\Z^{m+2})^\vee @>{\hD}>> \L^\vee\oplus\Z @>>> 0.
\end{CD}
\]
The fan of $E_j$ is contained in $N_{\R}\oplus\R$. 
The following generators of the 1-cones 
\[
\hb_i=(b_i,0) \ \text{ for } \ 1\leq i \leq m; \quad 
\hb_{m+1}=(\vec{0},1);\quad 
\hb_{m+2}=(b_j,-1),
\]
yield the fan sequence for $E_j$ 
\[
\begin{CD} 
0 @>>> \L\oplus \Z @>>> 
\Z^{m+2} @>{\hb}>> N\oplus\Z @>>> 0 
\end{CD}
\]
which is dual to the divisor sequence above. 
We set 
\[
p_0 := (\vec{0},1) = \hD_{m+1} = \hD_{m+2} \in H^2(E_j). 
\]
Under the splitting \eqref{eq:splitE}, 
a nef integral basis $\{p_1,\dots,p_r\}$ 
of $H^2(X)$ can be lifted to a nef integral basis 
$\{p_0,p_1,\dots,p_r\}$ of $H^2(E_j)$. 
We have 
\[
C_{E_j} = C_X + \R_{>0} p_0, \quad 
c_1(E_j) = c_1(X) + p_0. 
\]
The following result is immediate.

\begin{lemma}
\label{lem:Ej-nef} 
If $-K_X$ is nef then for all $j$, $-K_{E_j}$ is nef.
\end{lemma}

\subsection{I functions and the mirror maps.} 

We now recall Givental's
mirror Theorem \cite[Theorem 0.1]{Gi-A-98}. 
Let $\{p_0,p_1,\dots p_r\}$ be a  
nef basis of $H^2(E_j)$ as above. 
The I-function of $X$ is the 
$H^*(X)$-valued function: 
\[
I_X(y,z)=e^{\sum_{i=1}^r p_i\log y_i/z} 
\sum_{d\in\NE(X)_\Z} 
\prod_{i=1}^m \left( \frac{\prod_{k=-\infty}^0 (D_i+kz)}{
    \prod_{k=-\infty}^{\lan D_i, d\ran}(D_i+kz)} \right)
y^d. 
\]
Note that all but finitely many factors in the infinite products cancel.
Here $y^d = y_1^{\lan p_1,d\ran} \cdots 
y_r^{\lan p_r, d\ran}$. 
Similarly the I-function of $E_j$ is the 
$H^*(E_j)$-valued function 
\begin{equation}
\label{eq:IE}
I_{E_j}(y,z)=e^{\sum_{i=0}^r p_i\log y_i/z} 
\sum_{\beta\in \NE(E)_\Z}
\prod_{i=1}^{m+2} 
  \left(
    \frac{\prod_{k=-\infty}^0
      (\hD_i+kz)}{ \prod_{k=-\infty}^{\lan \hD_i,
        \beta \ran}(\hD_i+kz)}
  \right) y^\beta,
\end{equation}
where $y^\beta = y_0^{\lan p_0,\beta\ran} 
y_1^{\lan p_1, \beta\ran} \cdots y_r^{\lan p_r, \beta\ran}$.  

\begin{theorem}[Givental \cite{Gi-A-98}]
\label{t:mirror} 
Let $X$ be a toric manifold with $-K_X$ nef. 
Then we have 
\[
  I_X(y,z)=J_X(q,z)
\]
under an invertible change of variables of the form 
\begin{equation}
\label{eq:mirror_X}
    \log q_i= \log y_i + g_i(y_1, \dots, y_r), \quad 
    i=1,\dots,r 
\end{equation}
where $g_i(y)$ is a power series in $y_1,\dots,y_r$ 
which is homogeneous of degree zero 
with respect to the degree $\deg y^d = 2 \lan c_1(X), d\ran$ 
and $g_i(0)=0$.
\end{theorem}

\begin{definition}
The coordinates $y_1,\dots,y_r$ are called 
the \textbf{mirror coordinates} of $X$. 
They are asymptotically the same as 
$q_1,\dots,q_r$ as $y \to 0$, in the sense that 
$q_i =y_i + \text{higher order terms}$. 
\end{definition}

Because $c_1(E_j)$ is nef, we can apply this   
mirror theorem to $E_j$. 
Hence we have 
\[
I_{E_j}(y,z) = J_{E_j}(q,z) 
\]
under a change of variables 
\begin{equation} 
\label{eq:mirror}
\log q_i = \log y_i + g_i^{(j)}(y_0,y_1,\dots,y_r), 
\quad i=0,\dots, r. 
\end{equation}
\begin{lemma} 
\label{l:mirrormaps} 
The function $g_i^{(j)}$ does not depend on $y_0$. 
Moreover we have 
\[
g_i^{(j)}(y_0,y_1,\dots,y_r) = g_i(y_1,\dots,y_r), 
\quad i=1,\dots,r.   
\] 
This means that the mirror maps for $X$ and $E_j$ 
coincide for $q_1,\dots,q_r$. 
\end{lemma} 
\begin{proof} 
The functions $g_i^{(j)}$ appear as the 
coefficients of $z^{-1}$ 
in the expansion of $I_{E_j}$ 
(see \cite[p.145]{Gi-A-98}): 
\begin{equation} 
\label{eq:IEj_mirrormap}
I_{E_j}(y,z) = e^{\sum_{i=0}^r p_i \log y_i/z} 
\left(1 + z^{-1} 
\sum_{i=0}^r g_i^{(j)}(y) p_i +  
O(z^{-2}) \right). 
\end{equation} 
The functions $g_i$ are determined by $I_X$ similarly. 
Using $\iota^* \hD_j = D_j$, we can see that 
\begin{equation} 
\label{eq:IEj_IX} 
\iota^* I_{E_j} \Big |_{y_0=0} = I_X. 
\end{equation} 
Note that the restriction to $y_0=0$ in the left-hand side 
is well defined because $\iota^* p_0 = 0$. 
These facts imply that 
\[
g_i^{(j)}(0,y_1,\dots,y_r) = g_i(y_1,\dots,y_r), \quad 
i=1,\dots, r. 
\]
On the other hand, $\deg y_0 = 2 \lan c_1(E_j), \sigma_0\ran 
= 2$ and all other monomials appearing in $g_i^{(j)}$ 
have nonnegative degree (since $c_1(E_j)$ is nef). 
Thus the homogeneous series $g_i^{(j)}$ 
of degree zero does not depend on $y_0$. 
\end{proof}

\subsubsection{Batyrev relations and elements} 
Let $y_1,\dots,y_r$ be the mirror coordinates of a 
toric manifold $X$ with $-K_X$ nef.  
Set $\Q[y^{\pm}]=\Q[y_1^\pm, \dots , y_r^\pm]$. 
Batyrev's quantum ring is a $\Q[y^\pm]$-algebra 
generated by the variables $w_1,\dots, w_m$ 
corresponding to the toric divisors $D_1,\dots,D_m$   
subject to the following two types of relations: 
\begin{align}
\label{eq:bat-rel}
\begin{split}  
\text{(multiplicative):}& \qquad 
\prod_{j:\lan D_j, d\ran >0} w_j^{\lan D_j, d\ran} 
= y^d \prod_{j:\lan D_j,d\ran <0} 
w_j^{- \lan D_j, d\ran}, \quad d\in H_2(X,\Z); \\ 
\text{(linear):}& \qquad \sum_{j=1}^m c_j w_j =0 \quad 
\text{whenever} \quad 
\sum_{j=1}^m c_j D_j =0, \quad c_j \in \Q. 
\end{split} 
\end{align} 
where $y^d = \prod_{i=1}^r y_i^{\lan p_i, d\ran}$. 
We refer to these relations  
as \textbf{Batyrev relations}. 
By the divisor sequence \eqref{eq:divs}, the 
linear relations can be written in 
the form: 
\begin{equation} 
\label{eq:linearBat}
\sum_{j=1}^m \lan v, b_j \ran w_j = 0, \quad 
v \in M. 
\end{equation}

\begin{remark}\label{re:inv}
Since $X$ is compact, 
there exist positive integers $c_1,\dots,c_m$ 
such that $c_1b_1+ \cdots +c_m b_m =0$. 
Then by the fan sequence \eqref{eq:fans} 
we have $d\in H_2(X,\Z)$ such that 
$\lan D_i, d\ran = c_i >0$. 
This gives a relation $\prod_{i=1}^m w_i^{c_i} = y^d$. 
Therefore the variables $w_i$ are invertible in 
the Batyrev ring. 
\end{remark}

\begin{definition} \label{def:batelem} 
Let $X$ be a toric manifold with $-K_X$ nef. 
We can regard $p_i\in H^2(X)$ as corresponding to the 
logarithmic vector field $q_i (\partial/\partial q_i)$. 
We introduce an element $\tp_i \in H^2(X) \otimes 
\Q\llbracket y_1,\dots,y_r \rrbracket$ which corresponds to 
$y_i (\partial/\partial y_i)$ as  
\[
\tp_i = \sum_{k=1}^r 
\frac{\partial \log q_k}{\partial \log y_i} p_k 
= p_i + O(y). 
\] 
Recall from Equation 
\eqref{eq:div-basis} that $D_j = \sum_{i=1}^r m_{ij} p_i$. 
We define the \textbf{Batyrev element} associated to 
$D_j$ as 
\[
\tD_j=\sum_{i=1}^r m_{ij} \tp_i = D_j + O(y). 
\]
\end{definition}

\begin{prop}\label{p:bat-rel}
The Batyrev elements $\tD_1,\dots,\tD_m$ satisfy both 
the multiplicative and linear Batyrev relations 
\eqref{eq:bat-rel} for $w_j = \tD_j$. 
\end{prop}
\begin{proof} 
We use the fact \cite{Gi-Eq96, Gi-A-98,CoKa-Mi99} 
that if the J-function satisfies the differential 
equation 
$R(q_i, zq_i (\partial/\partial q_i),z) J_X(q,z)=0$, 
then we have a relation 
$R(q_i, p_i \bullet,0) 1 =0$ in the small quantum 
cohomology ring. 
We can easily show that the I-function of $X$ 
satisfy the differential equation $R_d I_X(y,z) =0$ 
for $d\in H_2(X,\Z)$ where 
\[
R_d = \prod_{j: \lan D_j,d\ran>0} 
\prod_{k=0}^{\lan D_j, d\ran -1} 
\left(\cD_j - k z\right) 
- y^d \prod_{j: \lan D_j, d \ran<0} 
\prod_{k=0}^{-\lan D_j, d\ran -1} 
\left(\cD_j - k z \right) 
\] 
and  
$\cD_j = z \sum_{i=1}^r m_{ij} y_i (\partial/\partial y_i)$. 
From the mirror theorem, we know that the J-function 
satisfies the corresponding differential equation 
under the mirror change of coordinates. 
The multiplicative Batyrev relations follow from this 
and the fact above. 
It is obvious that $\tD_j$'s satisfy the linear relations. 
\end{proof} 

\begin{remark}%\label{r:bat-sei-diff}
%Lemma \ref{p:bat-rel} ensures that the elements $\tD_j$
%satisfy both the multiplicative and the linear Batyrev
%relations. However 
Seidel elements satisfy the multiplicative
relations \eqref{eq:Batyrev_relation} in the $q$-coordinates, 
as proved in \cite[Proposition 5.2]{McTo-To06}, 
but do not necessarily satisfy the linear relations. 
\end{remark}

\subsection{Seidel elements in terms of mirror maps.}
In this paragraph we will assume that $-K_X$ is nef. 
Since $-K_{E_j}$ is also nef, 
we can expand the I-function of $E_j$ in $z^{-1}$ 
as follows (cf. \eqref{eq:IEj_mirrormap}). 
\begin{equation}
\label{eq:I-y}
I_{E_j}(y,z)= e^{\sum_{i=0}^r p_i \log y_i/z}
\left
    (1+ z^{-1}\sum_{i=0}^rg_i^{(j)}(y) p_i +
     z^{-2} \sum_{n=0}^2 G_n^{(j)}(y)y_0^n + 
 O(z^{-3}) 
\right)
\end{equation}
%This expression differs from the $J_{E_j}$ function only by the
%coefficient of $z^{-1}$, therefore this coefficient is exactly the
where $g_i^{(j)}(y)$ is the mirror map of $E_j$ 
in Equation \eqref{eq:mirror}. 
The coefficients $G_0^{(j)},G_1^{(j)},G_2^{(j)}$ 
are power series in $y_1,\dots,y_r$ 
taking values in 
$H^4({E_j}),H^2({E_j}),H^0({E_j})$ respectively. 
Under the coordinate change 
$\log q_i = \log y_i + g_i^{(j)}(y)$, we can 
rewrite $I_{E_j}(y,z)$ as 
\begin{align*}
e^{\sum_{i=0}^r p_i\log q_i/z} 
\left(
1+ z^{-2}\left(G_0^{(j)} -\frac{1}{2} 
\left(\textstyle\sum_{i=0}^r g_i^{(j)} p_i\right)^2 
+  G_1^{(j)} y_0 +G_2^{(j)} y_0^2 \right) + O(z^{-3}) 
\right). 
\end{align*} 

\begin{lemma}
\label{l:seidelmirror}
The Seidel element $S_j$ associated
to the toric divisor $D_j$ is given by 
\[
S_j(q_0,\dots ,q_r)=\iota^*(G_1^{(j)}(y_1,\dots y_r) y_0),
\]
under the mirror transformation \eqref{eq:mirror} 
for $E_j$. 
\end{lemma}

\begin{proof}
By Proposition \ref{p:SeidelfromJ}, 
the Seidel element 
$\tS_j$ is the coefficient of $q_0/z^2$ in 
$\exp(-\sum_{i=0}^r p_i \log q_i /z) J_{E_j}(q,z)$. 
The result follows from $J_{E_j}(q,z)=I_{E_j}(y,z)$. 
\end{proof}

We now digress to reinterpret $G_1^{(j)}$. 
A straightforward computation shows the 
following lemma.  
\begin{lemma} 
The I-function of $E_j$ 
satisfies the differential equation
\[
z\frac{\partial}{\partial y_0} 
\left(y_0 \frac{\partial}{\partial y_0}
\right)
I_{E_j} = 
\left(
\sum_{i=1}^r m_{ij} \left(y_i \frac{\partial}{\partial
  y_i}\right)-y_0\frac{\partial}{\partial y_0}
\right)  I_{E_j} .
\]
where $(m_{ij})$ is the matrix appearing in Equation 
\eqref{eq:div-basis}. 
\end{lemma}
% \begin{proof} 
% Comparing the derivatives in both sides 
% the result follows.  
% \end{proof} 

Using the expansion for $I_{E_j}$ in
Equation \eqref{eq:I-y} we obtain
\begin{align}
\notag 
&z\frac{\partial}{\partial y_0} 
\left(y_0\frac{\partial}{\partial y_0} \right) I_{E_j} 
=\notag \\
& = \frac{\partial}{\partial y_0} 
\left( 
e^{\sum_{i=0}^r p_i \log y_i/z} 
\left( 
p_0 + z^{-1} \left( 
\sum_{i=0}^r g_i^{(j)} p_i p_0 + 
\sum_{n=0}^2 G_n^{(j)} n y^n_0 \right) 
+    O(z^{-2}) \right) \right) \notag\\ 
\label{eq:batel1}
&=  
e^{\sum_{i=0}^r p_i \log y_i/z} \left(
    z^{-1} \sum_{n=1}^2 G_n^{(j)} n^2 y^{n-1}_0 +
    O( z^{-2}) \right)
\end{align}
where we used $p_0^2 =0$. 
On the other hand 
\begin{align} 
\nonumber 
\iota^* &\left(\sum_{i=1}^r m_{ij} 
y_i \frac{\partial}{\partial y_i} - 
y_0\frac{\partial}{\partial y_0} \right) I_{E_j} 
= \left(\sum_{i=1}^r m_{ij} y_i 
\frac{\partial}{\partial y_i} - 
y_0 \frac{\partial}{\partial y_0} 
\right) 
\iota^*I_{E_j} \\ 
\nonumber 
& \qquad = 
\left(\sum_{i=1}^r 
m_{ij} y_i \frac{\partial}{\partial y_i} \right) 
\left(I_X+ O(y_0)\right) 
\quad \text{by Equation \eqref{eq:IEj_IX}}  
\\ 
\nonumber 
& \qquad =  
\left(\sum_{i=1}^r m_{ij} y_i 
\frac{\partial}{\partial y_i} \right) 
\left(J_X+ O(q_0)\right) 
\qquad 
\text{by mirror Theorem \ref{t:mirror}} 
\\ 
\nonumber 
& \qquad = 
\left( 
\sum_{i=1}^r  \sum_{k=1}^r m_{ij}
\frac{\partial \log q_k} 
{\partial \log y_i} 
\frac{\partial}{\partial \log q_k} 
\right) 
e^{\sum_{i=1}^r p_i \log q_i/z} 
\left( 1 + O(z^{-2}) + O(q_0)\right) \\
\label{eq:batel}
& \qquad =e^{\sum_{i=1}^r p_i\log q_i/z} 
\left(z^{-1} \tD_j +O(z^{-2})+ O(q_0) \right)
\end{align}
where $\tD_j$ is the Batyrev element of $D_j$. 
Here $y_0,\dots,y_r$ and $q_0,\dots q_r$ are 
related by the mirror map \eqref{eq:mirror} 
for $E_j$, but we know by 
Lemma \ref{l:mirrormaps} that 
the mirror maps of $X$ and $E_j$ coincide 
for $q_1,\dots,q_r$. 
Comparing \eqref{eq:batel} and \eqref{eq:batel1} 
we get the following result.

\begin{lemma}\label{l:aux2}
  $\tD_j(y)=\iota^*G_1^{(j)}(y_1,\dots,y_r)$.  
\end{lemma}

Noting $S_j = q_0 \tS_j$ and 
the effect of the mirror map 
$\log q_0 = \log y_0 + g_0^{(j)}(y)$, 
we obtain the desired expression of 
Seidel elements from 
Lemma \ref{l:aux2} and Lemma \ref{l:seidelmirror}. 

\begin{theorem}
\label{th:batsei}
The Seidel element $\tS_j$ and Batyrev element $\tD_j$ are
related by 
\[ 
\tS_j(q_1,\dots, q_r)= 
\exp\left(-g_0^{(j)}(y_1,\dots,y_r)\right) 
\tD_j(y_1,\dots, y_r)
\] 
under the mirror transformation \eqref{eq:mirror_X} 
of $X$.  
\end{theorem}

\begin{remark}
The theorem above and Remark \ref{re:inv} show the invertibility
of the Seidel elements, obtaining the result of Seidel
%\ref{t:inv-sei} 
in our particular toric setting. 
\end{remark}

% We now want to refine the above results. 
% Write $S_j=\tS_j q_0$ (here $q_0=q^{\sigma_0}$), then
% \begin{equation}\label{eq:seibat}
%   \tS_j (q_1,\dots, q_j) = \left( \frac{y_0}{q_0} \right)
%   \tD_j(y_1,\dots, y_r) = \exp(-g_0(y_1,\dots,y_r)) 
%   \tD_j(y_1,\dots y_r)
% \end{equation}

% Note that the hypergeometric series $g_0(y_1,\dots,y_r)$ does depend
% on $j$, since it is part of the mirror transformation of $E_j$. One
% can explicitly compute it using Eq \eqref{eq:I-y}, \[ I_E(y,z)=
% e^{p\log y/z} \left (1+ \frac{1}{z}\sum_{i=0}^rg_i(y) p_i 
% + O(\frac{1}{z^2}) \right), \]

% and Equation \eqref{eq:divE}, we obtain an explicit formula for $g_0$ in
% terms of the toric divisors

% In the following two lemmas, we give 
% $\tD_j$ and $g_0^{(j)}$ as explicit 
% hypergeometric series in the mirror 
% coordinates $y$ of $X$. 

We shall calculate $g_0^{(j)}$ and $\tD_j$ 
as explicit hypergeometric series in the 
mirror coordinates $y$. 
It is not hard to see the following lemma. 

\begin{lemma} 
\label{lem:productfactor-asymptotics}
About the product factors appearing 
in the I-function $I_{E_j}$ \eqref{eq:IE}, 
we have  
\begin{align*} 
\prod_{i=1}^{m+2} 
\frac{\prod_{k=-\infty}^0 (\hD_i+kz)}
{ \prod_{k=-\infty}^{\langle \hD_i, \beta \rangle}(\hD_i+kz)} 
= C_\beta 
z^{-\sum_{i=1}^{m+2} \langle \hD_i,\beta\rangle - 
\sharp \{i:\langle \hD_i,\beta\rangle <0 \} } 
\prod_{i:\langle \hD_i,\beta\rangle <0} \hD_i 
& + \text{h.o.t.} 
\end{align*} 
where h.o.t.\ means higher order terms in $z^{-1}$ and 
\begin{equation} 
\label{eq:cst-asympt} 
C_\beta = 
\prod_{i:\langle \hD_i,\beta\rangle <0} 
(-1)^{-\langle \hD_i,\beta\rangle -1}(-\langle \hD_i,\beta\rangle -1)!
\cdot 
\prod_{i:\langle \hD_i,\beta \rangle \ge 0} (\langle \hD_i,\beta\rangle!)^{-1}. 
\end{equation} 
\end{lemma} 

\begin{lemma}
\label{l:g0} 
The coefficient $g_0^{(j)}$ is given by
\begin{equation}
\label{eq:g0}
g_0^{(j)}(y_1,\dots,y_r) 
=\sum_{\substack{\lan c_1(X), d\ran =0\\
        \lan D_j, d\ran<0\\
        \lan D_i, d\ran \geq
        0,\, \forall i\neq j}}
    \frac{(-1)^{\lan D_j, d\ran } 
     \left(- \lan D_j, d\ran -1\right)!}
{\prod_{i\neq j}\lan D_i, d\ran!}y^d.
\end{equation}
\end{lemma}

\begin{proof}
% Recall that our underlying assumption is that $-K_X$ is nef.
We want to investigate the coefficient of $z^{-1}$ in the
power series expansion \eqref{eq:IE} of $I_{E_j}$. 
% Therefore we need to know the
% leading term of the fraction
%   \[\frac{\prod_{k=-\infty}^0 (\hD_i+kz)}
% { \prod_{k=-\infty}^{\langle \hD_i, \beta \rangle}(\hD_i+kz)}\]
% in Equation \eqref{eq:IE} when expanded as a power series of
% $z^{-1}$. When $\langle \hD_i,\beta \rangle>0$, 
% it is not hard to see that the leading term is 
% \[
% \frac{1}{\langle \hD_i, \beta \rangle!}  
%  z^{-\langle \hD_i, \beta \rangle}.
% \] 
% If $\langle \hD_i, \beta \rangle=0$ the coefficient is clearly 1. 
% Similarly if $\langle \hD_i, \beta \rangle =-1$, 
% it is just $\hD_i$. 
% If $\langle \hD_i, \beta \rangle <-1$. 
% By expanding the series again we find that the
% leading coefficient is $(-1)^{-\langle \hD_i, \beta \rangle -1} \hD_i
% (-\langle \hD_i,\beta \rangle -1)! z^{-(\langle \hD_i, \beta \rangle +1)}$.
% 
% Thus, the coefficient of $z^{-1}$ in the product 
%   \begin{equation}\label{eq:aux2}
%   \prod_{i=1}^{m+2} 
%   \left(
%     \frac{\prod_{k=-\infty}^0
%       (\hD_i+kz)}{ \prod_{k=-\infty}^{\langle \hD_i,
%         \beta \rangle}(\hD_i+kz)}
%   \right)
%   \end{equation}
% is determined by all the effective classes $\beta \in \NE(E)_\Z$,
% $\sum_{i=1}^{m+2} \langle \hD_i, \beta \rangle=0$, 
% and $\langle \hD_i, \beta \rangle <0$ 
% for exactly one $i$.  
By Lemma \ref{lem:productfactor-asymptotics}, 
the summand indexed by $\beta\in \NE(E_j)_\Z$ 
contributes to the coefficient of $z^{-1}$ if 
$\sum_{i=1}^{m+2} \langle \hD_i, \beta\rangle + 
\sharp\{i:\langle \hD_i, \beta\rangle <0\}\le 1$.  
Since $-K_{E_j} = \sum_{i=1}^{m+2} \hD_i$ is nef 
(Lemma \ref{lem:Ej-nef}), 
this happens only in the following three cases: 
\begin{itemize} 
\item $\sum_{i=1}^{m+2} \langle \hD_i,\beta\rangle =0$ and 
$\sharp \{i: \langle \hD_i,\beta\rangle<0\} =0 $; 

\item $\sum_{i=1}^{m+2} \langle \hD_i,\beta\rangle =1$ 
and $\sharp\{i : \langle \hD_i,\beta\rangle<0\} =0$; 

\item $\sum_{i=1}^{m+2} \langle \hD_i,\beta\rangle =0$ 
and $\sharp\{i : \langle \hD_i,\beta\rangle<0\} =1$. 
\end{itemize} 
In the first case, we have $\langle \hD_i,\beta\rangle=0$ for 
all $i$ and so $\beta=0$. This contributes nothing to 
the coefficient of $z^{-1}$. 
The second case does not happen because in this case 
$\beta$ has to satisfy 
$\langle \hD_i,\beta \rangle=0$ except for one $i$, 
and this implies $\beta=0$. 
In the third case, $\beta$ has to be a fibre class 
from $\NE(X)_\Z$ (i.e.\ $\langle p_0, \beta \rangle =0$) 
because $\sum_{i=1}^{m+2} \hD_i = -K_X + p_0$ 
and $-K_X$, $p_0$ are nef. 
Therefore the coefficient of $z^{-1}$ in $I_{E_j}$ 
is the sum of 
\begin{equation} 
\label{eq:coeff-zinv} 
C_d \prod_{i:\langle \hD_i,d\rangle<0} \hD_i, \quad 
\text{where $C_d$ is the constant 
in \eqref{eq:cst-asympt}} 
\end{equation} 
over all the fibre classes 
$d\in \NE(X)_\Z$ such that 
$\sum_{i=1}^{m+2} \langle \hD_i, d \rangle = 
\sum_{i=1}^m \langle D_i, d\rangle =0$ and 
$\langle \hD_i,d\rangle =\langle D_i, d \rangle <0$ for 
exactly one $i$ from $\{1,\dots,m\}$. 
(Note that $\langle D_{m+1}, d\rangle = \langle D_{m+2},d\rangle=0$.)

Now $g_0^{(j)}$ is the coefficient corresponding to $p_0$. 
Among the divisors $\hD_1,\dots,\hD_m$, 
$\hD_j = (D_j,-1)=D_j - p_0$ is the only one which 
contains $p_0$. 
Therefore the terms of the form \eqref{eq:coeff-zinv} which 
contribute to $g_0^{(j)}$ are those with $d\in \NE(X)_\Z$ 
for which $\langle D_j,d\rangle <0$, $\langle D_i,d\rangle \ge 0$ 
for $i\neq j$ and $\sum_{i=1}^m \langle D_i, d\rangle =0$. 
% These last two do not contribute since 
% $\langle \hD_{m+1},d\rangle = \langle \hD_{m+2},d\rangle =0$ 
% for $d\in \NE(X)_\Z$. 
% Therefore the only terms that can contribute are 
% $C_d \hD_j$ 
% with fibre classes $d\in \NE(X)_\Z$ 
% for which $\langle D_j,d\rangle <0$, 
% $\langle D_i, d\rangle \ge0$ 
% for $i\neq j$ and 
% $\sum_{i=1}^m \langle D_i, d\rangle =0$. 
%   By the splitting of effective curve classes in Equation \eqref{eq:split2}
%   and Lemma \ref{},  only the effective 
%   fibre classes $d\in \NE(X)_\Z$ for which
%   $\lan D_j, d\ran <0$ and $\lan D_i, d\ran \geq 0, i\neq j$
%   contribute. 
% The sum of the contributions from such $d$ 
% gives  the formula \eqref{eq:g0}. 
% The coefficient of $I_{E_j}$ 
% corresponding to such classes is exactly
%   \[
%   \sum_{\substack{c_1(X)\cdot d=0\\
%         D_j\cdot d<0\\
%         D_i\cdot d\geq
%         0, i\neq j}}
%     \frac{(-1)^{D_j\cdot d}(-D_j\cdot d-1)!}{\prod_{i\neq
%         j}(D_i\cdot d)!}y^d,
%   \]
%   which finishes the proof.
\end{proof}

\begin{lemma}
The Batyrev element $\tD_j$ is given by 
\label{l:tDj}
\[
\tD_j = D_j - \sum_{i=1}^m D_i 
\sum_{
\substack{ \lan c_1(X), d\ran =0 \\
\lan D_i, d\ran <0 \\
\lan D_k, d\ran \ge 0, \, \forall k\neq i }} 
(-1)^{\lan D_i, d\ran}  
\frac{\left(-\lan D_i, d\ran -1\right)!} 
{\prod_{k\neq i, j} \lan D_k, d\ran !} 
y^d.   
\] 
\end{lemma} 
\begin{proof} 
By the same calculation leading to 
Equation \eqref{eq:batel}, we find 
\begin{align*} 
\left( 
\sum_{i=1}^r m_{ij} y_i \frac{\partial}{\partial y_i}\right) I_X(y,z) 
& = e^{\sum_{i=1}^r p_i \log q_i/z} 
\left(z^{-1} \tD_j + O(z^{-2}) \right)  \\ 
& = e^{\sum_{i=1}^r p_i \log y_i/z} 
\left(z^{-1} \tD_j + O(z^{-2}) \right). 
\end{align*} 
The conclusion follows from a calculation 
similar to the previous lemma. 
\end{proof} 

\subsection{Example} 
\label{subsec:BatSeiEx} 
Let $X=\P(\cO_{\P^1}\oplus\cO_{\P^1}(-2))$
be the second Hirzebruch surface. 
Let $p_1,p_2$ be the nef basis of
$H^2(X)$ that are Poincar\'{e} dual to the fibre and 
the infinity section. 
The divisor matrix is
\[
(m_{ij}) = 
\left[
\begin{array}{rrrr} 
0 & -2 & 1 & 1\\
1 & 1 & 0 & 0
\end{array}
\right],
\]
That is 
\begin{equation*}
D_1=p_2, \ D_2=p_2-2p_1, \ D_3=D_4=p_1. %\label{eq:ex1}
\end{equation*}
The $q$ coordinates
are
\[
q_1=e^{\PD(D_2)}, \ q_2 = e^{\PD(D_3)}.
\]
The mirror transformation of $X$ is well known 
(cf. \cite[p.146]{Gi-A-98}); it is given by 
\begin{align*}
  y_1&=\frac{q_1}{(1+q_1)^2},\\
  y_2&=q_2(1+q_1).
\end{align*}
%
% The vector fields 
% \(p_i=q_i\frac{\partial}{\partial q_i},p_i'=
% y_i\frac{\partial}{\partial y_i}\) associated to $p_i$ in the 
% $q$ and $y$ variables respectively, are related by 

The classes $\tp_i$, $i=1,2$ corresponding to 
the vector fields $y_i (\partial/\partial y_i)$ are 
\begin{align*}
  \tp_1&= \frac{1+q_1}{1-q_1} p_1- \frac{q_1}{1-q_1} p_2,\\ 
  \tp_2&=p_2.
\end{align*}
The Batyrev elements are 
\begin{align*} 
  \tD_1=\tp_2,\
  \tD_2=\tp_2-2 \tp_1,\
  \tD_3=\tD_4=\tp_1.
\end{align*}
Hence 
\begin{align*}
\tD_1&=D_1, \\
\tD_2&= \frac{1+q_1}{1-q_1} D_2, \\
\tD_3&=D_3- \frac{q_1}{1-q_1} D_2, \\
\tD_4&=D_4- \frac{q_1}{1-q_1} D_2.
\end{align*}
The correction term $g_0^{(j)}$ appears only for $j=2$. 
By Theorem \ref{th:batsei}, the Seidel element associated to 
$D_2$ is just given by
\[
\tS_2(q_1, q_2)=\exp\left(-g_0^{(2)}(y_1,y_2)\right)
\tD_2(y_1, y_2).
\]
where the term $g_0^{(2)}$ 
is the sum over all effective classes $d$ such
that $\lan c_1(X),d\ran=0$, 
$\lan D_j, d\ran <0, \lan D_i, d\ran \geq 0, i\neq j$ as
in Equation \eqref{eq:g0}. These are just the classes
$d=d_1\PD(D_2)$, for $d_1\geq 0$. So we have 
$\lan D_3, d\ran =\lan D_4,d\ran =d_1, 
\lan D_1,d\ran =0$ and $\lan D_2,d\ran =-2d_1$. Thus
\[
g_0(y)= \sum^\infty_{d_1=1} \frac{(2d_1-1)!}{(d_1!)^2} y_1^{d_1}.
\]
(This has already appeared in other places 
e.g.\ \cite[p.393]{CoKa-Mi99} or \cite[p.146]{Gi-A-98}.) 
Using the formulas above, 
it is not hard to see that
\[
\exp(-g_0)=(1+q_1)^{-1},
\]
and thus $\tS_2=(1+q_1)^{-1} \tD_2$. All other
Seidel elements \( \tS_i \) agree with the Batyrev
elements \( \tD_i\).
Therefore we have,
\begin{align*}
  \tS_1&=D_1,\\
  \tS_2&=\left(\frac{1}{1-q_1}\right)D_2,\\
  \tS_3&=D_3-\left( \frac{q_1}{1-q_1}\right)D_2,\\
  \tS_4&=D_4-\left( \frac{q_1}{1-q_1}\right)D_2.
\end{align*}
This computation agrees with the one in McDuff-Tolman
\cite{McTo-To06}. 

It is easy to check that the Batyrev relations are compatible in the
two coordinate systems
\begin{align*}
  \tD_1^{d_2}\tD_2^{d_2-2d_1} \tD_3^{d_1}
  \tD_4^{d_1} &= y_1^{d_1} y_2^{d_2},\\
  \tS_1^{d_2}\tS_2^{d_2-2d_1} \tS_3^{d_1}
  \tS_4^{d_1} &= q_1^{d_1} q_2^{d_2}.
\end{align*}

%%%%%%
\section{Reconstruction of mirror maps}
\subsection{When the correction term vanishes}
We continue to assume that $X$ is a toric manifold 
with $-K_X$ nef. 
Let $\Sigma$ denote its fan. 
If the toric divisor $D_j$ is nef, then a direct
computation from Equation \eqref{eq:g0} shows that the correction
coefficient $g_0^{(j)}$ is trivial. 
In such case the Seidel element and
Batyrev element agree. 
However the nef condition is too restrictive. 
We show that $g_0^{(j)}$ vanishes if and only if 
the restriction of $-K_X$ to $D_j$ is big, i.e. 
the image of $D_j$ under the map $\phi_{|-m K_X|}\colon 
X \to \P(H^0(X,-mK_X))$ 
has the same dimension as $D_j$ for sufficiently 
big $m>0$.

\begin{definition}
The \textbf{fan polytope} 
$P\subset N\tensor \R$ of $X$ is the 
convex hull of the integral primitive generators 
$b_1, \dots b_m$ of 1-cones of the fan $\Sigma$. 
\end{definition}

Since $X$ is a compact toric manifold with $-K_X$ nef, we have 
\begin{lemma}\label{l:orig}
The fan polytope $P$ of $X$ contains the origin in its interior and
every vector $b_i$ is on the boundary of $P$.
\end{lemma}

\begin{proposition} 
\label{p:vanishing} 
The following are equivalent: 
\begin{enumerate} 
\item The correction term 
$g_0^{(j)}$ in \eqref{eq:g0} vanishes;   
\item The primitive generator $b_j$ is 
a vertex of the fan polytope $P$; 
\item The anticanonical divisor 
$-K_X$ is big on $D_j$, i.e.\ $(-K_X)^{n-1}\cdot D_j >0$.  
\end{enumerate}  
\end{proposition} 

For the proof, we use the following notation 
and lemma.  

\begin{notation}
  For each cone $\sigma\in \Sigma$ we let $F(\sigma)$ denote the
  minimal face of the fan polytope $P$ which contains the collection
  $\sigma(1)$ of primitive generators $b_i\in \sigma$. By abuse of
  notation we write $F(b_i)$ to denote the face $F(\R_{\geq 0} b_i)$.
\end{notation}
\begin{lemma}\label{l:aux4}
  Let $\sigma$ be a cone in $\Sigma$. Suppose that $d\in H_2(X)$
  satisfies $\lan c_1(X) , d\ran=0 $ and 
  \[\lan D_i, d \ran\geq 0 \quad \text{if} \quad 
  b_i\notin \sigma.\]
  Then $d\in \NE(X)$ and 
  \[
  \lan D_i, d\ran=0 \quad \text{if} \quad b_i \notin F(\sigma).
  \]
\end{lemma}
\begin{proof}
  Take a support function $h\colon N_\R\to \R$ 
  such that $P$ is contained in
  the half-space $\{v\in N_\R|h(v)\leq 1\}$ and such that $F(\sigma)=P\cap
  h^{-1}(1)$. We have the following relation
  \[
  0=\sum_{i=1}^m \lan D_i, d\ran b_i.
  \]
  By evaluating $h$ we have
  \begin{align*}
    0=\sum_{i=1}^m \lan D_i, d\ran h(b_i) = 
\sum_{b_i\in F(\sigma)}
    \lan D_i, d \ran + 
\sum_{b_i\notin F(\sigma)} \lan D_i, d \ran
    h(b_i) \\ 
    \leq \sum_{b_i\in F(\sigma)} \lan D_i, d \ran +
    \sum_{b_i\notin F(\sigma)} \lan D_i, d \ran=0,
  \end{align*}
where the second inequality follows from the fact that $h(b_i)<1$ and
$\lan D_i,d\ran\geq 0$, if $b_i\notin F(\sigma)$. Because of the
zeroes in each hand-side, the inequality is an equality. Therefore
$\lan D_i,d\ran h(b_i)=\lan D_i, d\ran$, if $b_i\notin
F(\sigma)$. This in turn implies that $\lan D_i, d\ran=0$, if $b_i\notin
F(\sigma)$.
\end{proof}

\begin{proof}[Proof of Proposition \ref{p:vanishing}]
We first prove the equivalence of (ii) and (iii). 
By \cite[Lemma 9.3.9]{CoLi-To10}, the bigness of 
$(-K_X)|_{D_j}$ translates to the fact that 
the face $F_j^*$ of the dual polytope $P^*$ 
has the maximal dimension $n-1$, where 
\begin{align*} 
P^* &= \{ v\in M_\R \,:\, \lan v, b_i \ran \ge -1,\, \forall i \},\\  
F_j^* &= \{ v\in P^*\, :\, \lan v, b_j \ran =-1\}. 
\end{align*} 
This is equivalent to $b_j$ being a vertex of $P$. 

Next we prove the equivalence of (i) and (ii). 
  Suppose $b_j$ is not a vertex of $P$. 
Then $b_j$ is in the relative interior of $F(b_j)$. 
Since $F(b_j)$ is convex, there exist $b_{i_1},
  \dots ,b_{i_k}$ on $F(b_j)$ and nonnegative 
constants $c_1, \dots, c_k$ such
  that such that $b_{i_s}\neq b_j$ for all  $s$ and
\begin{gather}
\label{eq:badrel}     
c_1 b_{i_1} + \dots + c_k b_{i_k} - b_j=0, \\ 
\label{eq:prod}
c_1+\dots +c_k -1=0.
\end{gather}
By the fan sequence \eqref{eq:fans}, 
the relation \eqref{eq:badrel} gives an element in $d\in\NE(X)$ such
  that $\lan D_j,d\ran=-1, \lan D_i,d\ran\geq 0, i\neq j$. 
By Equation \eqref{eq:prod}, $\lan c_1(X),d \ran=0$. 
Such $d$ contributes to the sum in Equation \eqref{eq:g0} 
and $g_0^{(j)}\neq 0$. 

% To show the converse we need first the following result.
% \begin{proposition}
%   If $b_i$ is a vertex of $P$, 
%   then the associated divisor $D_i$ is positive.
% \end{proposition}
% \begin{proof}

Suppose that $b_j$ is a vertex of $P$ 
and $g_0^{(j)} \neq 0$. 
Then by Equation \eqref{eq:g0} 
there exists $d\in \NE(X)$ such that
$\lan D_j,d\ran<0$, $\lan D_i,d\ran\geq 0, i\neq j$ 
and $\lan c_1(X),d \ran=0$. 
By Lemma \ref{l:aux4}, we
know that $\lan D_i, d\ran=0 $, if $b_i\notin F(b_j)$. 
However, $b_j$ is a vertex of $P$, 
which means that $F(b_j)$ contains only $b_j$. 
Therefore $\lan D_i, d\ran=0 $ for all $i\neq j$. 
Since the divisors $\{D_i\,:\, i\neq j\}$ 
span  $H^2(X,\Q)$ it follows that $d=0$, which contradicts $\lan
  D_i,d\ran<0$. 
\end{proof} 

% \begin{corollary}\label{c:good-ver}
%   The toric divisor $D_i$ is positive if and only if $b_i$ is a vertex of
%   the fan polytope $P$.
% \end{corollary}

\begin{corollary}
Let $n$ be the dimension of $X$ 
and $r$ be the Picard number. 
Out of $m=n+r$ correction terms $g_0^{(j)}$, 
at least $n+1$ vanish. 
In other words, there are at most 
$r-1$ non-vanishing correction terms. 
\end{corollary}

\begin{proof}
Any convex polyhedron with
non-empty interior in an $n$-dimensional vector space has at least
$n+1$ vertices. The result follows.
\end{proof}

\subsection{Reconstruction}
By Proposition \ref{p:bat-rel} the elements $\tD_j$ satisfy
both the multiplicative and the linear Batyrev relations. However
Seidel elements only satisfy the multiplicative relations. The
reconstruction of the mirror transformation is based on this
observation. We now have all we need to establish 
the proof of Theorem \ref{t:intro3}.

% \begin{theorem}\label{t:bats} 
% \marginpar{Statement changed} 
%   The Batyrev elements can be uniquely reconstructed from the Seidel
%   elements. More precisely, given the Seidel elements
%   $\tS_1,\dots,\tS_m$, there exist unique elements 
% $\tD_1,\dots,\tD_m$ in $QH^*(X)$ satisfying the following conditions:
%   \begin{enumerate}
%   \item\label{cc:1}  $\tD_j =H_j(q) \tS_j$ for some power series
%     $H_j(q)\in \Lambda_X.$
%   %\item\label{cc:2} $H_j(q)\in 1+\mathfrak{m}$.
%   \item\label{cc:3} $H_j(q)= 1$ if $(-K_X)^{n-1}\cdot D_j >0$.  
% \item\label{cc:4} The set $\tD_j$ satisfies the linear relations in 
% Equation \eqref{eq:bat-rel}. 
%   \end{enumerate}
% This in particular shows that the Seidel elements reconstruct
% the mirror coordinates. 
% \end{theorem}

\begin{proof}[Proof of Theorem \ref{t:intro3}] 
By all our statements above, Batyrev elements satisfy 
the conditions (i) -- (iii). 
So we only need to show the uniqueness property. 
The linear relations (iii) are equivalent to the identity
(cf.\ Equation \eqref{eq:linearBat}) 
\[
\sum_{j=1}^m b_j\tensor \tD_j=0
\]
in the tensor product $N_\Q\tensor H^2(X;\Lambda_X)$. 
Substituting $\tD_j$ with $H_j \tS_j$ 
and using the property (ii), we have 
\[
\sum_{b_j: \text{non-vertex}} H_j (b_j \otimes \tS_j)  
= -\sum_{b_j: \text{vertex}} b_j \otimes \tS_j. 
\]
Here we used the fact that $(-K_X)^{n-1}\cdot D_j>0$ 
is equivalent to $b_j$ being a vertex of the fan 
polytope (Proposition \ref{p:vanishing}). 
Since $\tS_j = D_j + O(q)$, to see the uniqueness of $H_j$, 
it suffices to show that 
$\{b_j\otimes D_j\,:\, b_j: \text{non-vertex}\}$
is linearly independent. 
Suppose we have a linear relation 
\begin{equation} 
\label{eq:rel_bjotimesDj}
\sum_{b_j: \text{non-vertex}} h_j (b_j\otimes D_j) =0. 
\end{equation} 
%
% Suppose that there is
% another element $\tD_j'=H_j'(q)\tS_j, j=1,\dots, m$
% which satisfy conditions \eqref{cc:1}-\eqref{cc:4}. Choose any
% K\"ahler class $\eta$ of $X$ and define the following valuation on
% $\Lambda_X$,
% \[
% v\left( \sum_d c_d q^d \right):=\min\{\lan \eta, d\ran |c_d\neq 0\}.
% \]
%
% Since both families of elements $\tD_j$, $\tD_j'$
% both satisfy the linear relations, we have 
% \begin{equation}
% \label{eq:compare}
% 0=\sum_{i=1}^m b_j\tensor 
% (\tD_j- \tD_j') = 
% \sum_{i=1}^m b_j\tensor ( H_j- H_j') \tS_j.
% \end{equation}
%
% We will compare the terms with lowest valuation $v_0>0$. Set 
% \[
% H_j- H_j'= K_j + (\text{terms with valuation bigger than } v_0),
% \ \ K_j=\sum_{\lan \eta, d\ran=v_0} c_{j,d} q^d.
% \]
%
% the series $K_j=0$ if $D_j$ is positive, since in this case
% $H_j=H_j'=1$. Because of the asymptotic expansion 
% $\tS_j=D_j+\cO(q)$ we find that the lowest order terms of Equation
% \eqref{eq:compare} give
%
% \begin{equation}
%   \label{eq:compare2}
%   0=\sum_{j:D_j\text{ not positive}}b_j\tensor K_jD_j.
% \end{equation}
%
% We claim that all the coefficients $K_j$ is zero. To see this it
% suffices to show that $\{b_j\tensor D_j|D_j \text{ not positive}\}$ is
% linearly independent. 
If $b_j$ is not a vertex, the face $F(b_j)$ of 
the fan polytope $P$ contains $b_j$ 
in its relative interior. 
Since $F(b_j)$ is a convex hull 
of vertices $b_{i_1},\dots,b_{i_k}$ of $P$, 
we have a relation 
\[
c_1 b_{i_1} + \dots + c_k b_{i_k} - b_j=0
\]
for some rational numbers $c_s$. 
This relation gives an element $d\in H_2(X)$ such that 
$\lan D_j,d\ran = -1$, $\lan D_{i_s}, d\ran = c_s$ and 
$\lan D_l, d\ran =0$ when $l$ is none of $j, i_1,\dots,i_k$. 
Contracting Equation \eqref{eq:rel_bjotimesDj} 
with $\id \otimes d$, we get 
\[
0=-h_j b_j.   
\] 
Thus $h_j=0$ for all $j$ for which $b_j$ is not a vertex.
This completes the proof. 
\end{proof}

\subsection{Examples} 
We compute the Batyrev and the Seidel elements 
in several examples 
and illustrate the reconstruction of mirror maps. 

\subsubsection{Hirzebruch surface $\F_2$} 
We revisit the example in Section \ref{subsec:BatSeiEx}. 
The fan of $\F_2$ is given by 
the following primitive vectors of 1-cones: 
\[
b_1 = (0,-1),\ b_2 = (0,1), \ b_3 = (-1,1), \ 
b_4 = (1,1). 
\]
Here $b_1,b_3,b_4$ are vertices of the fan polytope 
and thus $\tS_j = \tD_j$ for $j=1,3,4$. 
The mirror coordinates $y_1,y_2$ 
can be reconstructed from these 
Seidel elements, via the formulas 
\[
\tS_1 = \sum_{i=1}^2 
\frac{\partial \log q_i }{\partial \log y_2} p_i, 
\quad 
\tS_3 = \tS_4 = 
\sum_{i=1}^2 
\frac{\partial \log q_i}{\partial \log y_1} p_i. 
\]

\subsubsection{Crepant resolution of $\P^3/\Z^2$}
% \label{sss:e1}
Let $X=\P(\cO(2,-2)\oplus \cO)$ be 
the $\P^1$-bundle over $\P^1\times \P^1$. 
Collapsing the zero and the infinity section 
to $\P^1$ yields $\P^3/\Z_2$, 
where $\Z_2$ acts on $\P^3$ by
\[
[z_1:z_2:z_3:z_4]\mapsto [z_1:z_2:-z_3:-z_4].
\]
The toric variety 
$\P^3/\Z_2$ has transversal $A_1$ singularities along two 
$\P^1$'s and $X$ is its crepant resolution.

The fan of $X$ is given by 
\begin{gather*}
b_1=(1,0,-1), \ b_2=(-1,0,-1),\  b_3=(0,1,1), \ b_4=(0,-1,1),\\
 b_5=(0,0,1),\  b_6=(0,0,-1).
\end{gather*}
The vertices of the fan polytope are 
$b_1,b_2,b_3,b_4$. 
% One can easily check the divisor relations 
% \begin{gather*}
%   D_1=D_2, \ D_3=D_4, \ D_3+D_4+D_5=D_1+D_6+D_2. 
% \end{gather*}
% The relations amongst the $b_i's$ with Chern number zero are
% \begin{align*}
%   d_1&=(0,0,1,1,-2,0)\\
%   d_2&=(1,1,0,0,0,-1)
% \end{align*}
% and with positive Chern number
% \begin{align*}
%   c_1&=(0,0,0,0,1,1)\\
%   c_2&=(0,0,1,1,0,2)\\
%   c_3&=(1,1,0,0,2,0),
% \end{align*}
% and thus 
% \[
% c_2=d_1 +2c_1, c_3=d_2+2c_1.
% \]
The Mori cone is spanned by 
the three homology classes 
$\gamma_1,\gamma_2,\gamma_3 
\in H_2(X,\Z)$ with the intersection matrix 
\[
\left(\lan D_i, \gamma_j \ran \right)^{\rm T}=
\begin{pmatrix}
  0&0&1&1&-2&0\\
  1&1&0&0&0&-2\\
  0&0&0&0&1&1
\end{pmatrix} 
\] 
where each row vector gives a relation 
amongst $b_1,\dots,b_6$. 
The classes $\gamma_1,\gamma_2,\gamma_3$ define 
a $\Z$-basis for $H_2(X,\Z)$.  
Let $p_1,p_2,p_3 \in H^2(X,\Z)$ denote 
the dual basis and let 
$q_1,q_2,q_3$ denote the corresponding 
Novikov variables. 

% Note that only $D_5,D_6$ are
% non-positive, and $D_1,D_3,c_1(X)$ generate the whole of $H_2(X)$.
% The intersection matrix is given by 
% The mirror maps are given by
% \[
% \sum_{i=1}^3 p_i \log q_i 
% = \sum_{i=1}^3 p_i\log y_i + g(y_1)D_5 + g(y_2)D_6,
% \]
% where 
% \[
% g(y)=\log\left(\frac{1+\sqrt{1-4y}}{2}\right),
% \]
% and therefore 
% \begin{align*}
%   \log q_1&=\log y_1 -2g(y_1),\\
%   \log q_2&=\log y_2 -2g(y_2),\\
%   \log q_3&=\log y_3 +g(y_1) +g(y_2).
% \end{align*}
\noindent
\textbf{Mirror coordinates.} 
\begin{align*}
  y_1&=\frac{q_1}{(1+q_1)^2},\\
  y_2&=\frac{q_2}{(1+q_2)^2},\\
  y_3&=q_3(1+q_1)(1+q_2)
\end{align*}
% The Jacobian of the mirror map is
% \begin{equation*}
% \left(\frac{\partial \log y_a}{\partial \log q_b}\right)=
%   \begin{pmatrix}
%     \frac{1-q_1}{1+q_1}&0&0\\
%     0&\frac{1-q_2}{1+q_2}&0\\
%     \frac{q_1}{1+q_1}&\frac{q_2}{1+q_2}&1
%   \end{pmatrix};
% \end{equation*}
% with inverse
% \begin{equation*}
% \left(\frac{\partial \log q_a}{\partial \log y_b}\right)=
%   \begin{pmatrix}
%      \frac{1+q_1}{1-q_1}&0&0\\
%     0&\frac{1+q_2}{1-q_2}&0\\
%     -\frac{q_1}{1-q_1}&-\frac{q_2}{1-q_2}&1.
%   \end{pmatrix}
% \end{equation*}
%The columns of this matrix are the elements $p'_i$, and thus 
\textbf{The Batyrev and Seidel elements.} 
\begin{align*}
\tS_1 &= \tS_2 = 
\tD_1=\tD_2=\tp_2 
=\frac{1+q_2}{1-q_2}p_2-\frac{q_2}{1-q_2}p_3, \\
\tS_3 &= \tS_4 = 
\tD_3 = \tD_4 = \tp_1 = 
\frac{1+q_1}{1-q_1}p_1-\frac{q_1}{1-q_1}p_3,\\
\tS_5 & = \frac{1}{1+q_1} \tD_5, \quad 
\tD_5=-2\tp_1+\tp_3=\frac{1+q_1}{1-q_1}D_5,\\
\tS_6 & = \frac{1}{1+q_2} \tD_6, \quad 
\tD_6 =-2\tp_2+\tp_3=\frac{1+q_2}{1-q_2}D_6.
\end{align*}
\textbf{Reconstruction from Seidel elements.} 
The reconstruction in this case 
becomes a little more involved. 
Since $\tS_1,\dots,\tS_4$ have no correction terms, 
we have the relations 
\begin{equation} 
\label{eq:3dim_Sei-mirrortrans} 
\tS_1 = \tS_2 = \sum_{i=1}^3 
\frac{\partial \log q_i}{\partial \log y_2} p_i, 
\quad 
\tS_3 = \tS_4 = 
\sum_{i=1}^3 \frac{\partial \log q_i}{\partial \log y_1} 
p_i  
\end{equation} 
but these are not enough to determine the mirror 
coordinates $y_1,y_2,y_3$. 
We also need to require that the mirror map 
is homogeneous, i.e.\ the Euler vector field 
is preserved 
\begin{equation} 
\label{eq:Eulerpreserved}
2 y_3 \frac{\partial}{\partial y_3} 
= 2 q_3 \frac{\partial}{\partial q_3}. 
\end{equation} 
The Equations \eqref{eq:3dim_Sei-mirrortrans}, 
\eqref{eq:Eulerpreserved} can reconstruct 
the mirror coordinates $y_1,y_2,y_3$. 

In general, the method of the reconstruction illustrated 
here works if $c_1(X)$ together with the divisors 
$D_j$ for which $b_j$ is a vertex of the fan polytope 
span $H^2(X)$. 

\subsubsection{Crepant resolution of 
$(\P^1\times \P^1)/\Z_2$}
Let $\Z_2$ act on $\P^1\times \P^1$ by 
\[
([z_1, z_2], [w_1,w_2]) \mapsto 
([z_1,-z_2], [w_1,-w_2]).  
\]
The quotient $(\P^1\times \P^1)/\Z_2$ has 
four isolated singular points of type $A_1$. 
Let $X$ denote the minimal resolution of 
$(\P^1\times \P^1)/\Z_2$. 
It is given by the complete regular fan 
(Figure \eqref{fig:box}) spanned by 
\begin{gather*} 
b_1=(-1,1), \ b_2=(1,1), \ b_3=(1,-1),\  b_4=(-1,-1),\\
b_5=(0,1),\  b_6=(1,0),\ b_7=(0,-1),\ b_8=(-1,0).
\end{gather*} 
The vertices of the fan polytope are $b_1,\dots,b_4$. 

%%%%%%%
\begin{figure}
\begin{center}
\scalebox{.7}{
  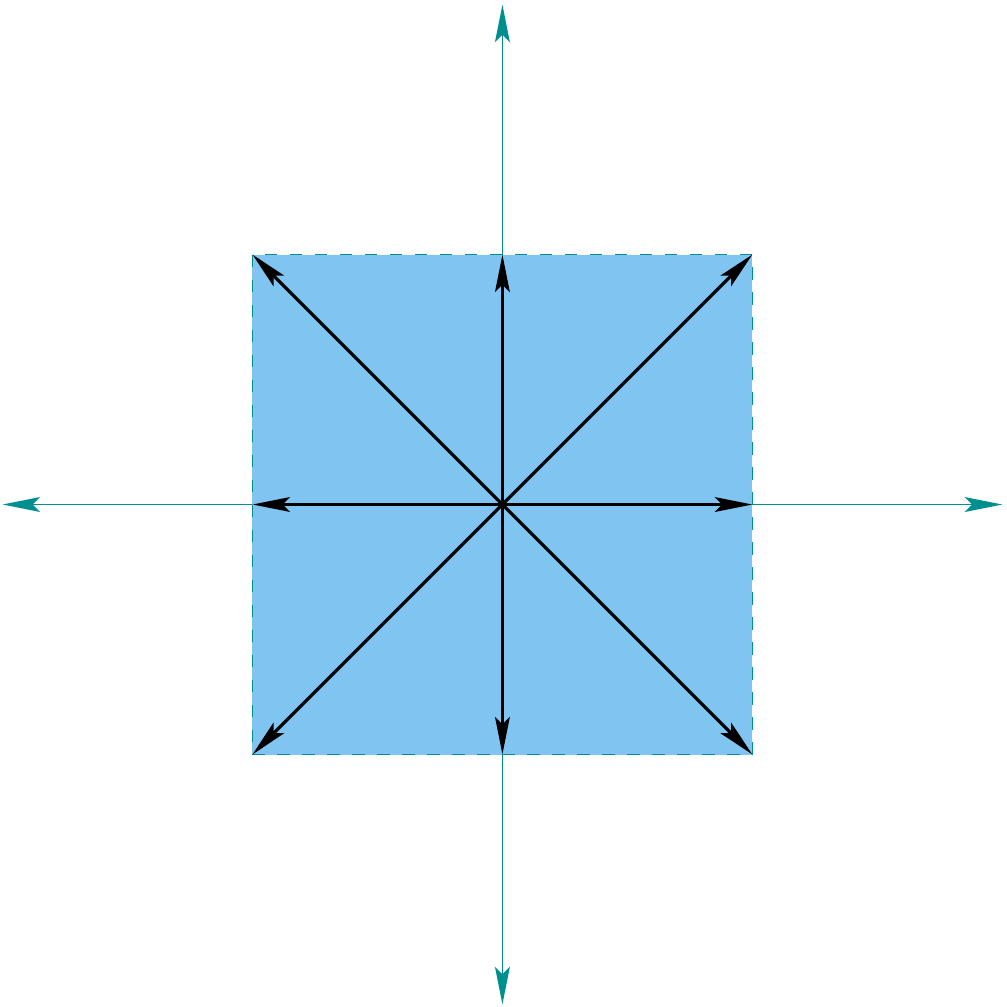
}
\end{center}
\caption{Crepant resolution of $(\P^1\times\P^1)/\Z_2$}
\label{fig:box}
\end{figure}
%
% The associated divisors $D_i$ satisfy:
% \begin{align*}
%   D_1+D_5+D_2=D_4+D_7+D_3\\
%   D_1+D_8+D_4=D_2+D_6+D_3.
% \end{align*}
% The positive divisors are clearly $D_1,D_2,D_3,D_4$ 
% while the non-positive ones are
% $D_5,D_6,D_7,D_8.$ 
% A relation 
% \[
% \sum_{i=1}^8 c_i b_i=0
% \]
% defines an element in $H^2(X)$ such that $\lan D_i,d\ran=c_i$. 
% We will denote $d$ by a column vector 
% \[
% d=\begin{pmatrix}
% c_1 \\ \vdots \\ c_8 
% \end{pmatrix}.  
% \]
Define $\gamma_1,\dots,\gamma_6\in H_2(X,\Z)$ 
to be the $\Q$-basis of $H_2(X)$ 
with the following intersection matrix: 
\[
\left(\lan D_i, \gamma_j \ran\right)^{\rm T}
= 
\begin{pmatrix} 
1 & 1 & 0 & 0 & -2 & 0 & 0 & 0 \\ 
0 & 1 & 1 & 0 & 0 & -2 & 0 & 0 \\
0 & 0 & 1 & 1 & 0 & 0 & -2 & 0 \\
1 & 0 & 0 & 1 & 0 & 0 & 0 & 2 \\ 
-2 & -1 & -2 & -1 & 2 & 2 & 2 & 2 \\ 
-1 & -2 & -1 & -2 & 2& 2& 2& 2 
\end{pmatrix} 
\]
Again each row vector gives a relation 
of $b_1,\dots,b_8$. 
In this case the Mori cone $\NE(X)$ is 
not simplicial, but is contained in 
the cone spanned by $\gamma_1,\dots,\gamma_6$. 
Let $p_1,\dots,p_6 \in H^2(X)$ be the dual 
basis of $\gamma_1,\dots,\gamma_6$ 
and let $q_1,\dots, q_6$ be the corresponding 
Novikov variables. Note that 
$\deg q_1 = \cdots = \deg q_4 =0$ 
and $\deg q_5 = \deg q_6 = 4$. 

\noindent 
\textbf{Mirror coordinates.}
\begin{equation*}
%\label{eq:eg3-m-inv}
  y_i =
  \begin{cases}
    \displaystyle\frac{q_i}{(1+q_i)^2}, &i =1,2,3,4\\
    q_i (1+q_1)^2(1+q_2)^2(1+q_3)^2(1+q_4)^2, &i=5,6.
  \end{cases}
\end{equation*}
From this we get the elements
\begin{align*}
\tp_i 
=\sum_{i=1}^6
\frac{\partial \log q_k}{\partial \log y_i} p_k 
=
\begin{cases}
  \displaystyle \frac{1+q_i}{1-q_i}p_i -\frac{2q_i}{1-q_i}p_5
  -\frac{2q_i}{1-q_i}p_6, &i=1,2,3,4.\\
\displaystyle  p_i, &i=5,6.
\end{cases}
\end{align*}
\textbf{The Batyrev and Seidel elements.} 
\begin{align*}
  \tS_1& =\tD_1= \tp_1+\tp_4-2\tp_5-\tp_6 \\
%   &=\frac{1+q_1}{1-q_1}p_1 + \frac{1+q_4}{1-q_4}p_4 -
%   \left(2+\frac{2q_1}{1-q_1}+\frac{2q_4}{1-q_4}\right) p_5
%   -\left(1+\frac{2q_1}{1-q_1}+\frac{2q_4}{1-q_4}\right) p_6\\
  &=\frac{1+q_1}{1-q_1}p_1 + \frac{1+q_4}{1-q_4}p_4 -
  \left(\frac{1+q_1}{1-q_1}+\frac{1+q_4}{1-q_4}\right) p_5
  -\left(\frac{1+q_1}{1-q_1}+\frac{2q_4}{1-q_4}\right) p_6\\
\tS_2&=\tD_2= \tp_1+\tp_2-\tp_5-2\tp_6\\
  &=\frac{1+q_1}{1-q_1}p_1 + \frac{1+q_2}{1-q_2}p_2 -
  \left(\frac{1+q_1}{1-q_1}+\frac{2q_2}{1-q_2}\right) p_5
  -\left(\frac{1+q_1}{1-q_1}+\frac{1+q_2}{1-q_2}\right) p_6\\
\tS_3&=\tD_3= \tp_2+\tp_3-2\tp_5-\tp_6\\
  &=\frac{1+q_2}{1-q_2}p_2 + \frac{1+q_3}{1-q_3}p_3 -
  \left(\frac{1+q_2}{1-q_2}+\frac{1+q_3}{1-q_3}\right) p_5
  -\left(\frac{1+q_2}{1-q_2}+\frac{2q_3}{1-q_3}\right) p_6\\
\tS_4&=\tD_4= \tp_3+\tp_4-\tp_5-2\tp_6\\
  &=\frac{1+q_3}{1-q_3}p_3 + \frac{1+q_4}{1-q_4}p_4 -
  \left(\frac{1+q_3}{1-q_3}+\frac{2q_4}{1-q_4}\right) p_5
  -\left(\frac{1+q_3}{1-q_3}+\frac{1+q_4}{1-q_4}\right) p_6\\
\tS_5& = \frac{1}{1+q_1} \tD_5, \quad 
\tD_5=-2\tp_1+2\tp_5+2\tp_6 =\frac{1+q_1}{1-q_1}
(-2p_1+2p_5+2p_6) =\frac{1+q_1}{1-q_1} D_5\\
\tS_6 & = \frac{1}{1+q_2} \tD_6, \quad 
\tD_6=-2\tp_2+2\tp_5+2\tp_6 
=\frac{1+q_2}{1-q_2}
(-2p_2+2p_5+2p_6)=\frac{1+q_2}{1-q_2} D_6 \\
\tS_7 &= \frac{1}{1+q_3} \tD_7, \quad 
\tD_7 =-2\tp_3+2\tp_5+2\tp_6=\frac{1+q_3}{1-q_3}
(-2p_3+2p_5+2p_6) =\frac{1+q_3}{1-q_3} D_7 \\ 
\tS_8 & = \frac{1}{1+q_4} \tD_8, \quad
\tD_8 =-2\tp_4+2\tp_5+2\tp_6=\frac{1+q_4}{1-q_4}
(-2p_4+2p_5+2p_6)=\frac{1+q_4}{1-q_4} D_8.
\end{align*}
% We know use as before Theorem \ref{th:batsei} and Equation
% \eqref{eq:g0} to compute the remaining Seidel elements for
% non-positive divisors.
% \begin{align}
%   \tS_i&=\exp\left( -\sum_{\substack{c_1(X)\cdot d=0\\
%         D_j\cdot d<0\label{eq:eg3:aux0}\\
%         D_i\cdot d\geq 0, i\neq j}} \frac{(-1)^{D_j\cdot d}(-D_j\cdot
%       d-1)!}{\prod_{i\neq j}(D_i\cdot d)!}y^d\right)\tD_i\\
% &=\frac{1}{1+q_{i-4}} \tD_i,\ i=5,6,7,8.\notag
% \end{align}
% Therefore 
% \begin{equation}\label{eq:eg3-aux1}
%   \tS_i=\tD_i, \ \ \ i=1, 2, 3, 4, \text{ and }\
%   \tS_i=\frac{1}{1-q_{i-4}}D_i, \ \ \  i=5,6,7,8. 
% \end{equation}

\noindent
\textbf{Reconstruction from the Seidel elements.} 
In this case, the divisors $D_1,\dots,D_4$ 
and $c_1(X)$ do not span $H^2(X)$ and 
the method in the previous example does no apply. 
Assume that we know the Seidel elements 
$\tS_1,\dots, \tS_8$ given above. 
We will reconstruct the Batyrev elements 
from these. 
We set $\tD_i :=\tS_i$ for $i=1,2,3,4$ 
and $\tD_i := H_i \tS_i$ for $i=5,6,7,8$ 
for some $H_i\in \Q\llbracket q\rrbracket$. 
The linear relation
\begin{align*}
\tD_1+\tD_5+\tD_2=\tD_4+\tD_7+\tD_3.
% \ \ \ \text{ thus}\\
% H_5 \tS_5-H_7 \tS_7= \tD_3+\tD_4-\tD_1-\tD_2.
\end{align*}
gives the equation 
\[
H_5 \tS_5-H_7 \tS_7= \tS_3+\tS_4-\tS_1-\tS_2.  
\]
% &=2\frac{1+q_3}{1-q_3}p_3
% -2\frac{1+q_1}{1-q_1}p_1 -\left( -2\frac{1+q_3}{1-q_3}
% -2\frac{1+q_1}{1-q_1}\right)p_5 -\left( 2\frac{1+q_3}{1-q_3}
% -2\frac{1+q_1}{1-q_1}\right)p_6\\
%
% \end{align*}
From the computation above, we have 
\begin{align*}
  -\frac{H_7}{1-q_3}D_7+ \frac{H_5}{1-q_1}D_5=
  -\frac{1+q_3}{1-q_3}D_7+ \frac{1+q_1}{1-q_1}D_5.
\end{align*}
Since $D_7$ and $D_5$ are linearly independent, 
it follows that $H_5=1+q_1$, $H_7=1+q_3$. 
Similarly one can solve for $H_6,H_8$. 

%%%%
\subsubsection{Crepant resolution of 
$\P(1,1,2)/\Z_2$} 
\label{eg: triang} 
Consider the $\Z_2$ action on the weighted 
projective space $\P(1,1,2)$ given by 
\[
[z_1,z_2, z_3] \mapsto [-z_1,z_2, -z_3]. 
\]
The quotient $\P(1,1,2)/\Z_2$ has three singular 
points, two of which are type $A_1$ and 
the other is of type $A_3$. 
Let $X$ be the minimal resolution of $\P(1,1,2)/\Z_2$.  
% We now give another example of a NEF toric manifold 
% whose positive divisors and
% $c_1(X)$ do not span $H^2(X;\Q)$. 
The fan of $X$ is given by (see Figure \ref{fig:count})
\begin{gather*} 
b_1 = (-2,1), \ b_2 = (2,1), \ b_3 = (0,-1) \\ 
b_4 = (-1,1), \ b_5 = (0,1), \ b_6 = (1,1), \ 
b_7 = (1,0), \ b_8 = (-1,0) 
\end{gather*} 
The Picard number of $X$ is 6. 
The vectors $b_1,b_2,b_3$ are vertices 
of the fan polytope. 

% Explicit relations that show that they
% are non-positive are the following,
% \begin{align*}
%   &b_1-2b_4+b_5=0\\
%   &b_4-2b_5+b_6=0\\
%   &b_5-2b_6+b_2=0\\
%   &b_2-2b_7+b_3=0\\
%   &b_3-2b_8+b_1=0.
% \end{align*}
%% 
% The remaining divisors $D_1,D_2,D_3$ are positive, and it is obvious that
% these together with $c_1(X)$ cannot span the $6$-dimensional space
% $H^2(X)$.
%%
\begin{figure}
\begin{center}
\scalebox{.7}{
  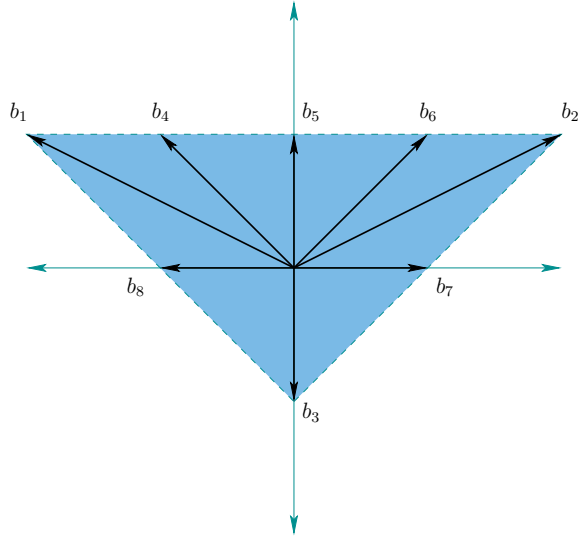
}
\end{center}
\caption{Fan polytope of a smooth NEF toric manifold 
for which the divisors corresponding to vertices 
and $c_1$ do not span $H^2(X)$. This gives a 
crepant resolution of $\P(1,1,2)/\Z_2$.}
\label{fig:count}
\end{figure}

Let $\{\gamma_1,\dots,\gamma_6\}$ 
be the basis of $H_2(X,\Q)$
such that 
\[
\left(\lan D_i, \gamma_j\ran \right)^{\rm T}  
= \begin{pmatrix} 
1 & 0 & 0 & -2 & 1 & 0 & 0 & 0 \\ 
0 & 0 & 0 & 1 & -2 & 1 & 0 & 0 \\ 
0 & 1 & 0 & 0 & 1 & -2 & 0 & 0 \\ 
0 & 1 & 1 & 0 & 0 & 0 & -2 & 0 \\ 
1 & 0 & 1 & 0 & 0 & 0 & 0 & -2 \\ 
-2 & -2 & -1 & 2 & -1 & 2&2& 2
\end{pmatrix} 
\] 
% 1 & 0 & 0 & 0 & 1 & -2 \\ 
% 0 & 0 & 1 & 1 & 0 & -2  \\
% 0 & 0 & 0 & 1 & 1 & -1 \\
% -2 & 1 & 0 & 0 & 0 & 2 \\ 
% 1 & -2 & 1 & 0 & 0 & -1 \\ 
% 0 & 1 & -2 & 0 & 0 & 2 \\ 
% 0& 0 & 0 & -2 & 0 & 2 \\ 
% 0 & 0 & 0 & 0 & -2 & 2 
% \end{pmatrix}   
% \]
The Mori cone $\NE(X)$ is not simplicial 
but is contained in the 
simplicial cone generated by $\gamma_1,\dots, \gamma_6$. 
(Here $\gamma_i \in \NE(X)$ for $1\le i\le 5$, 
but $\gamma_6 \notin \NE(X)$. Note also that 
$\gamma_1,\dots, \gamma_6$ do not form an integral basis.) 
Set the Novikov variable $q_i := q^{\gamma_i}$. 
Let $\{p_1,\dots,p_6\}\subset H^2(X,\Q)$ 
be the dual basis of $\{\gamma_1,\dots, \gamma_6\}$.  

\noindent
\textbf{Mirror coordinates.} 
The mirror coordinates for surface 
$A_n$ singularity resolutions are calculated in 
\cite[Appendix A, Proposition A.6]{CCIT}. 
The same method applies here, as the mirror map 
for the $A_3$ singularity resolution appears 
as part of the mirror map in this case. 
We only state the final result here. 

\begin{align*}
%\label{eq:mirrorcoord}
\begin{split}  
y_1 &= q_1 \frac{1+ q_2 + q_1 q_2 + q_2 q_3 
+ q_1 q_2 q_3 + q_1 q_2^2 q_3}
{(1+ q_1 + q_1 q_2 + q_1 q_2 q_3)^2}, \\
y_2 & = q_2 \frac{
(1+q_1 + q_1 q_2 + q_1 q_2 q_3)
(1+ q_3 + q_2 q_3 + q_1 q_2 q_3)}
{(1+q_2 +q_1 q_2 + q_2 q_3 + q_1 q_2 q_3 
+ q_1 q_2^2 q_3)^2}, \\ 
y_3 & = q_3 \frac{1+q_2 +q_1 q_2 +  q_2q_3
+ q_1 q_2 q_3 + q_1 q_2^2 q_3}
{(1+q_3 + q_2 q_3 + q_1 q_2 q_3)^2}, \\
y_4 & = \frac{q_4}{(1+q_4)^2}, \\
y_5 & = \frac{q_5}{(1+q_5)^2}, \\ 
y_6 & = q_6 
\frac{(1+q_1+q_1q_2+q_1q_2q_3)^2(1+q_3+q_2q_3+q_1q_2q_3)^2
(1+q_4)^2 (1+q_5)^2}
{1+q_2 +q_1 q_2 + q_2 q_3 + q_1 q_2 q_3 
+ q_1 q_2^2 q_3}.
\end{split} 
\end{align*} 

\noindent 
\textbf{The Batyrev and Seidel elements.}
Here we present a cohomology class $\sum_{i=1}^6 c_i p_i$ 
by the column vector $(c_1,c_2,\dots,c_6)^{\rm T}$.

\[
\tD_1 =\tS_1 = 
\begin{pmatrix} 
\frac{1+q_1-q_2-q_2q_3-q_1^2q_2+q_2^2q_3 
-q_1^2q_2q_3 + q_1^3q_2^2 q_3} 
{(1-q_1)(1-q_2)(1-q_1q_2)(1-q_2q_3)(1-q_1q_2q_3)} \\
-\frac{q_1(1-q_3-q_1q_2 +q_2^2 -q_1 q_2^2 
+ q_1 q_2 q_3 -q_2^3 q_3 + q_1 q_2^3 q_3)} 
{(1-q_1)(1-q_2)(1-q_3)(1-q_1q_2)(1-q_2 q_3)} \\ 
\frac{q_1q_2^2(1-q_2q_3 -q_1q_2q_3 +q_3^3 
-q_2q_3^3+ q_1 q_2^2 q_3^2 -q_1q_2q_3^3  + q_1 q_2^2 q_3^3)}
{(1-q_2)(1-q_3)(1-q_1q_2)(1-q_2q_3)(1-q_1q_2q_3)} \\
0 \\
\frac{1+q_5}{1-q_5} \\
\frac{1-3 q_5}{1-q_5} +  
\frac{q_1}{(1-q_1) ( 1-q_2 ) ( 1-q_2 q_3) }
- 
\frac{3}{( 1- q_1 )(1- q_1q_2) 
(1- q_1 q_2 q_3)}
-
\frac {2 q_1 q_2^2}{(1-q_2)(1-q_3)(1-q_1q_2)}
\end{pmatrix} 
\]
\[
\tD_2 = \tS_2 = 
\begin{pmatrix} 
\frac{q_2^2 q_3 (1 -q_1q_2+q_1^3 -q_1q_2q_3  
- q_1^3 q_2 - q_1^3 q_2 q_3 + q_1^2 q_2^2 q_3 + 
q_1^3q_2^2q_3)}
{(1-q_1)(1-q_2)(1-q_1q_2)(1-q_2 q_3)(1-q_1 q_2q_3)} \\
-\frac{q_3(1-q_1+q_2^2-q_2q_3 +q_1q_2q_3 
-q_2^2 q_3 -q_1 q_2^3 +q_1q_2^3q_3)}
{(1-q_1)(1-q_2)(1-q_3)(1-q_1q_2)(1-q_2q_3)} \\
\frac{1-q_2+q_3-q_1q_2+q_1q_2^2-q_2q_3^2
-q_1 q_2 q_3^2+q_1 q_2^2 q_3^3}
{(1-q_2)(1-q_3)(1-q_1 q_2)(1-q_2 q_3)(1-q_1q_2q_3)} \\ 
\frac{1+q_4}{1-q_4} \\ 
0 \\ 
\frac{1-3q_4}{1-q_4} 
+ \frac{q_3}{(1-q_2)(1-q_3)(1-q_1q_2)} 
- \frac{3}{(1-q_3)(1-q_2q_3)(1-q_1q_2q_3)} 
- \frac{2 q_2^2 q_3}
{(1-q_1)(1-q_2)(1-q_2q_3)}
\end{pmatrix} 
\]
\[
\tD_3 = \tS_3 = 
\begin{pmatrix}
0 \\ 
0 \\
0 \\ 
\frac{1+q_4}{1-q_4} \\ 
\frac{1+q_5}{1-q_5} \\ 
-\frac{2}{1-q_4} + \frac{1-3q_5}{1-q_5} 
\end{pmatrix} 
\]
\begin{align*} 
\tD_4 &= (1+q_1+ q_1q_2 + q_1 q_2 q_3) \tS_4, \\
\tS_4 &= 
\begin{pmatrix} 
-\frac{2-q_2 -q_1 q_2 - q_2 q_3 - q_1 q_2 q_3 + q_2^3 q_3 
+ q_1^2 q_2^2 q_3}{(1-q_1)(1-q_2) (1-q_1q_2)(1-q_2q_3) 
(1-q_1q_2 q_3)} \\
\frac{1+q_2-q_3 -2 q_1 q_2 + q_1 q_2 q_3 - q_2^3 q_3 
+ q_1 q_2^2 q_3 }{(1-q_1)(1-q_2) (1-q_3)(1-q_1 q_2) 
(1-q_2 q_3)}  \\ 
-\frac{q_2(1 - q_2 q_3 + q_3^2 - q_1 q_2 q_3 
- q_2 q_3^2 - q_1 q_2 q_3^2 + 2 q_1 q_2^2 q_3^2)} 
{(1-q_2)(1-q_3)(1-q_1 q_2)(1-q_2 q_3)(1-q_1 q_2 q_3)} \\ 
0 \\
0 \\ 
-\frac{1}{(1-q_1)(1-q_2)(1-q_2 q_3)} 
+ \frac{3}{(1-q_1)(1-q_1q_2)(1-q_1 q_2 q_3)} 
+ \frac{2q_2}{(1-q_2)(1-q_3)(1-q_1q_2)}  
\end{pmatrix} 
\end{align*} 
\begin{align*} 
\tD_5 &= (1+q_2 + q_1q_2 + q_2 q_3 +q_1q_2 q_3 + q_1 q_2^2 q_3) 
\tS_5 \\ 
\tS_5 &= 
\begin{pmatrix} 
\frac{1+q_1 -2 q_1 q_2 - 2 q_1 q_2 q_3 + q_1 q_2^2 q_3 
+ q_1^2 q_2^2 q_3} 
{(1-q_1)(1-q_2)(1-q_1 q_2) (1- q_2 q_3)(1-q_1 q_2 q_3)} \\ 
-\frac{2 -q_1 -q_3 - q_1 q_2 - q_2 q_3 + 2 q_1 q_2 q_3}
{(1-q_1)(1-q_2)(1-q_3)(1-q_1q_2)(1-q_2q_3)} \\ 
\frac{1+ q_3 - 2 q_2 q_3 - 2 q_1 q_2 q_3 + q_1 q_2^2 q_3 
+ q_1 q_2^2 q_3^2}
{(1-q_2)(1-q_3)(1-q_1q_2)(1-q_2q_3)(1-q_1q_2q_3)} \\ 
0 \\ 
0 \\
\frac{1}{(1-q_1)(1-q_2)(1-q_2 q_3)} 
- \frac{2}{(1-q_2)(1-q_3)(1-q_1q_2)} 
-\frac{3q_1}{(1-q_1)(1-q_1q_2)(1-q_1q_2q_3)}  
\end{pmatrix}  
\end{align*} 
\begin{align*}
\tD_6 & = (1+ q_3 + q_2 q_3 + q_1 q_2 q_3) \tS_6 \\ 
\tS_6 & =
\begin{pmatrix} 
-\frac{q_2(1+q_1^2 - q_1 q_2 - q_1^2 q_2 - q_1 q_2 q_3 
-q_1^2 q_2 q_3 + 2 q_1^2 q_2^2 q_3)}
{(1-q_1)(1-q_2)(1-q_1q_2)(1-q_2 q_3)(1-q_1 q_2 q_3)} \\ 
\frac{1 - q_1 + q_2 -2 q_2 q_3 - q_1 q_2^2 + q_1 q_2 q_3 
+ q_1 q_2^2 q_3}
{(1-q_1)(1-q_2)(1-q_3)(1-q_1q_2)(1-q_2 q_3)} \\
- \frac{2 - q_2 -q_1 q_2 - q_2 q_3 + q_1 q_2^2 
- q_1 q_2 q_3 + q_1 q_2^2 q_3^2}
{(1-q_2)(1-q_3)(1-q_1q_2)(1-q_2 q_3) (1-q_1 q_2 q_3)} \\ 
0 \\
0 \\ 
-\frac{1}{(1-q_2)(1-q_3)(1-q_1q_2)} 
+ \frac{3}{(1-q_3)(1-q_2q_3)(1-q_1q_2 q_3)} 
+ \frac{2 q_2} { (1-q_1)(1-q_2)(1-q_2 q_3)} 
% \\
% \frac{2 - 2 q_1- q_2 - q_2 q_3 + q_1^2 q_2 + q_1 q_2^2 
% - 3 q_1^2 q_2^2 - q_1^2 q_2 q_3 + q_1 q_2^2 q_3^2 
% + 2 q_1^2 q_2^3 q_3 + q_1^2 q_2^2 q_3^2 
% - 2 q_1^2 q_2^3 q_3^2}
% {(1-q_1)(1-q_2)(1-q_3)(1-q_1q_2)(1-q_2q_3)(1-q_1q_2q_3)}
\end{pmatrix}  
\end{align*} 
\begin{align*}
\tD_7 = (1+q_4) \tS_7, 
\quad 
\tS_7 = 
\begin{pmatrix}
0 \\ 
0 \\ 
0 \\ 
-\frac{2}{1-q_4} \\ 
0 \\
\frac{2}{1-q_4} 
\end{pmatrix}, 
\quad 
\tD_8 = (1+q_5) \tS_8, 
\quad 
\tS_8 = 
\begin{pmatrix} 
0 \\
0\\
0 \\
0 \\
-\frac{2}{1-q_5} \\ 
\frac{2}{1-q_5} 
\end{pmatrix}.   
\end{align*} 
\textbf{Reconstruction of Batyrev from Seidel.} 
Again the divisors $D_1,D_2,D_3$ and $c_1(X)$ 
do not span $H^2(X)$ in this case. 
We have the following 
linear relations: 
\begin{align*} 
2 \tD_1 + \tD_4 + \tD_8 &= \tD_6 + \tD_7 + 2 \tD_2, \\ 
\tD_1 + \tD_4 + \tD_5 + \tD_6 + \tD_2 &= \tD_3. 
\end{align*} 
Suppose we know only the Seidel elements 
$\tS_1,\dots,\tS_8$ given above. 
We can check (assisted by computer) that there exist 
unique functions $x, y, z,u,v$ of $q_i$'s 
which solves the linear equations: 
\begin{align*} 
2 \tS_1 + x \tS_4 + v \tS_8 &=
 z \tS_6 + u \tS_7 + 2 \tS_2, \\ 
\tS_1 + x \tS_4 + y \tS_5 + z \tS_6 + \tS_2 
& = \tS_3.  
\end{align*} 
Here $x,y,z,u,v$ are given by (as expected) 
\begin{align*} 
x & = 1+q_1 + q_1 q_2 + q_1 q_2 q_3 \\ 
y & = 1 + q_2 + q_1 q_2 + q_2 q_3 + q_1 q_2 q_3 
+ q_1 q_2^2 q_3 \\ 
z &= 1 + q_3 + q_2 q_3 + q_1 q_2 q_3 \\ 
u & = 1 + q_4 \\ 
v & = 1+ q_5.   
\end{align*} 
The Batyrev elements are given as: 
$\tD_i = \tS_i$ for $1\le i\le 3$ 
and $\tD_4 = x \tS_4$, $\tD_5 = y \tS_5$, 
$\tD_6 = z \tS_6$, $\tD_7 = u \tS_7$, 
$\tD_8= v \tS_8$. 
Then the Batyrev elements determine the 
mirror co-ordinates $y_1,\dots,y_6$. 

\subsubsection{Crepant resolution of $\P^2/\Z_3$} 
Consider the toric variety $X$ given by 
the fan spanned by the following vectors 
(see Figure \ref{fig:cubic}): 
\begin{gather*} 
b_1 = (0,1),\ b_2 = (1,0),\ b_3 = (1,-1), \ 
b_4 = (0,-1), \ b_5 = (-1,0), \ b_6 = (-1,1) \\ 
b_7 = (-1,2), \ b_8 = (2,-1), \ b_9 = (-1,-1) 
\end{gather*} 
Only $b_7,b_8,b_9$ are the vertices of the fan polytope. 
A cubic surface in $\P^3$ can degenerate 
to the singular toric variety $\P^2/\Z_3$ 
(with 3 nodes of type $A_2$) 
and $X$ is a minimal resolution of $\P^2/\Z_3$
where the $\Z_3$-action on $\P^2$ is 
given by 
$[z_1,z_2,z_3] \mapsto [z_1,\omega z_2, \omega^2 z_3]$.  
We can construct $X$ as a 6 times blowup 
of $\P^2$ at infinitely near points. 
Thus $X$ is deformation equivalent to 
a del Pezzo surface of degree 3 (i.e.\ cubic 
surface). 

\begin{figure}
\begin{center}
\scalebox{.7}{
  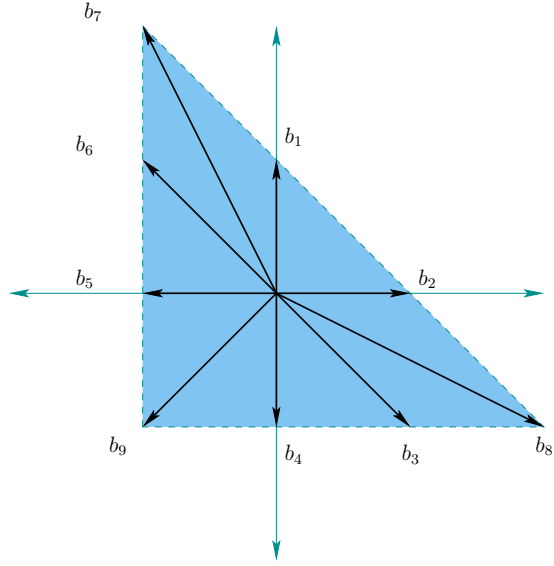
}
\end{center}
\caption{Toric degeneration of a cubic surface}
\label{fig:cubic}
\end{figure}

We take a basis $\{\gamma_1,\dots, \gamma_7\}$ of $H_2(X,\Q)$ 
such that the intersection matrix becomes 
\[
\left(\lan D_i, \gamma_j \ran\right)^{\rm T} 
=  
\begin{pmatrix} 
-2 & 1 & 0 & 0 & 0 & 0 & 1 & 0 & 0 \\ 
1 & -2 & 0 & 0 & 0 & 0 & 0 & 1 & 0 \\ 
0 & 0 & -2 & 1 & 0 & 0 & 0 & 1 & 0 \\ 
0 & 0 & 1 & -2 & 0 & 0 & 0 & 0 & 1 \\ 
0 & 0 & 0 & 0 & -2 & 1 & 0 & 0 & 1 \\
0 & 0 & 0 & 0 & 1 & -2 & 1 & 0 & 0 \\
2 & 2 & 2 & 2 & 2 & 2 & -3 & -3 & -3 
\end{pmatrix} 
\] 
% Here each row vector gives a relation of $b_i$'s and 
% a curve class. 
% Each column vector gives a toric divisor. 
% (The matrix here is the transpose of the one in 
% the hand-written notes.)
The Mori cone $\NE(X)$ is contained in the simplicial 
cone generated by $\gamma_i$'s. 
Note that $\gamma_1,\dots, \gamma_7$ are not a $\Z$-basis 
and that $\NE(X)$ itself is not simplicial.
We have $\gamma_i \in \NE(X)$, $1\le i \le 6$ and 
$\gamma_7 \notin \NE(X)$. 
As usual, we take $\{p_1,\dots, p_7\} \subset H^2(X,\Q)$ 
to be the dual basis of $\{\gamma_1,\dots, \gamma_7\}$ 
and $q_i := q^{\gamma_i}$ to be the Novikov variable. 

\noindent 
\textbf{Mirror coordinates.} 
\begin{align*} 
y_1 & = q_1 \frac{1+q_2 + q_1q_2}{(1+q_1 + q_1q_2)^2} \\ 
y_2 & = q_2 \frac{1+ q_1 + q_1q_2}{(1+q_2 + q_1q_2)^2} \\
y_3 & = q_3 \frac{1+ q_4 + q_3 q_4}{(1+q_3 + q_3 q_4)^2} \\ 
y_4 & = q_4 \frac{1+ q_3+ q_3 q_4}{(1+q_4+q_3 q_4)^2} \\
y_5 & = q_5 \frac{1+ q_6 + q_5 q_6}{(1+q_5 + q_5 q_6)^2} \\ 
y_6 & = q_6 \frac{1+ q_5 + q_5 q_6}{(1+q_6 + q_5 q_6)^2} \\ 
y_7 & = q_7 (1+q_1+ q_1q_2)^2 (1+q_2 + q_1 q_2)^2 
(1+q_3 + q_3q_4)^2  \\ 
& \quad \cdot (1+ q_4 + q_3 q_4)^2 
(1+ q_5 + q_5 q_6)^2 (1+ q_6 + q_5 q_6)^2 
\end{align*} 
\textbf{The Batyrev and Seidel elements.} 
\[ 
\tD_1 = (1+q_1 + q_1q_2) \tS_1, \quad 
\tS_1 = \frac{1}{\Delta_{12}} 
\begin{pmatrix} 
- 2 + q_2 + q_1 q_2 \\
1 + q_2 -2 q_1 q_2 \\ 
0 \\ 
0 \\
0 \\
0 \\
2(1 - 2 q_2 + q_1 q_2) 
\end{pmatrix}
\] 
\[
\tD_2 = ( 1+ q_2 + q_1 q_2) \tS_2, \quad 
\tS_2 = 
\frac{1}{\Delta_{12}} 
\begin{pmatrix} 
1 + q_1 - 2 q_1 q_2 \\ 
- 2 + q_1  + q_1q_2 \\ 
0 \\
0 \\
0 \\
0 \\
2(1 -2 q_1 + q_1 q_2) 
\end{pmatrix}
\]
%%%%%%%%%%%%%%%%%%%%
\[ 
\tD_3 = (1+q_3 + q_3q_4) \tS_3, \quad 
\tS_3 = \frac{1}{\Delta_{34}} 
\begin{pmatrix} 
0 \\ 
0 \\
- 2 + q_4 + q_3 q_4 \\
1 + q_4 -2 q_3 q_4 \\ 
0 \\
0 \\
2(1 - 2 q_4 + q_3 q_4) 
\end{pmatrix}
\] 
\[
\tD_4 = ( 1+ q_4 + q_3 q_4) \tS_4, \quad 
\tS_4 = 
\frac{1}{\Delta_{34}} 
\begin{pmatrix} 
0 \\
0 \\
1 + q_3 - 2 q_3 q_4 \\ 
- 2 + q_3  + q_3q_4 \\ 
0 \\
0 \\
2(1 -2 q_3 + q_3 q_4) 
\end{pmatrix}
\]
%%%%%%%%%%%%%
\[ 
\tD_5 = (1+q_5 + q_5q_6) \tS_5, \quad 
\tS_5 = \frac{1}{\Delta_{56}} 
\begin{pmatrix} 
0 \\ 
0 \\
0 \\
0 \\
- 2 + q_6 + q_5 q_6 \\
1 + q_6 -2 q_5 q_6 \\ 
2(1 - 2 q_6 + q_5 q_6) 
\end{pmatrix}
\] 
\[
\tD_6 = ( 1+ q_6 + q_5 q_6) \tS_6, \quad 
\tS_6 = 
\frac{1}{\Delta_{56}} 
\begin{pmatrix} 
0 \\
0 \\
0 \\
0 \\
1 + q_5 - 2 q_5 q_6 \\ 
- 2 + q_5  + q_5q_6 \\ 
2(1 -2 q_5 + q_5 q_6) 
\end{pmatrix}
\]
\[
\tD_7 = \tS_7 = \frac{1}{\Delta_{12}\Delta_{56}} 
\begin{pmatrix}
  (1+q_1 - q_2 -q_1^2 q_2) \Delta_{56}\\
  - q_1 (1 - q_1 q_2 + q_2^2 - q_1 q_2^2)\Delta_{56} \\
  0 \\
  0 \\
  - q_6(1+ q_5^2 - q_5 q_6 - q_5^2 q_6)\Delta_{12} \\
  (1- q_5 + q_6 -q_5 q_6^2) \Delta_{12}\\
  5-2(2-q_1-2q_2+q_1^2q_2)\Delta_{12}^{-1}-
  2(2-2q_5-q_6+q_5q_6^2)\Delta_{56}^{-1}
\end{pmatrix} 
\] 
\[
\tD_8 = \tS_8 = \frac{1}{\Delta_{12}\Delta_{34}}
\begin{pmatrix} 
-q_2(1 - q_1 q_2 + q_1^2 - q_1^2 q_2)\Delta_{34} \\ 
(1 - q_1 + q_2 - q_1 q_2^2)\Delta_{34} \\ 
(1 + q_3 -q_4 - q_3^2 q_4)\Delta_{12} \\ 
-q_3(1 - q_3 q_4 + q_4^2 - q_3 q_4^2)\Delta_{12} \\ 
0 \\ 
0 \\ 
5-2(2-2q_1-q_2+q_1q_2^2)\Delta_{12}^{-1}-
  2(2-q_3-2q_4+q_3^2q_4)\Delta_{34}^{-1}
\end{pmatrix} 
\]
\[
\tD_9 = \tS_9 
= \frac{1}{\Delta_{34}\Delta_{56}} 
\begin{pmatrix} 
0 \\ 
0 \\ 
- q_4(1 + q_3^2 - q_3 q_4 - q_3^2 q_4)\Delta_{56} \\ 
(1 - q_3 + q_4 - q_3 q_4^2)\Delta_{56} \\ 
(1 + q_5 - q_6 - q_5^2 q_6)\Delta_{34} \\ 
-q_5(1 - q_5 q_6 - q_5 q_6^2 + q_6^2) \Delta_{34} \\
5-2(2-2q_3-q_4+q_3q_4^2)\Delta_{34}^{-1}-
  2(2-q_5-2q_6+q_5^2q_6)\Delta_{56}^{-1} 
\end{pmatrix} 
\]
where $\Delta_{ij}=(1-q_i)(1-q_j)(1-q_iq_j)$. 

\noindent
\textbf{Reconstruction of Batyrev from Seidel.} 
We have the following linear relations: 
\begin{align*} 
2 \tD_7 + \tD_1 + \tD_6 &= \tD_9 + \tD_4 + \tD_3 + \tD_8 \\ 
2 \tD_8 + \tD_2 + \tD_3 &= \tD_9 + \tD_5 + \tD_6 + \tD_7 
\end{align*} 
Suppose we only know the Seidel elements 
$\tS_1,\dots, \tS_9$. 
We can check that the following linear 
equation for $x,y,z,w,u,v$ has a unique solution: 
\begin{align*}
2 \tS_7 + x \tS_1 + v \tS_6 &= \tS_9 + w\tS_4 + z\tS_3 + \tS_8 \\ 
2 \tS_8 + y\tS_2 + z \tS_3 &= \tS_9 + u \tS_5 + v \tS_6 + \tS_7 
\end{align*} 
% The solutions are 
% \begin{align*} 
% x &= 1 + q_1 + q_1 q_2 \\ 
% y &= 1 + q_2 + q_1 q_2 \\ 
% z & = 1 + q_3 + q_3 q_4 \\ 
% w & = 1 + q_4 + q_3 q_4 \\
% u & = 1 + q_5 + q_5 q_6 \\ 
% v & = 1 + q_6 + q_5 q_6 
% \end{align*} 
The Batyrev elements are given as 
$\tD_1 = x \tS_1$, $\tD_2 = y \tS_2$, $\tD_3 = z \tS_3$, 
$\tD_4 = w \tS_4$, $\tD_5 = u \tS_5$, $\tD_6 = v \tS_6$ 
and $\tD_i = \tS_i$ for $7\le i\le 9$. 
The mirror co-ordinates $y_i$ are determined 
by these. 
\begin{remark} 
It is interesting to compare the Givental-Hori-Vafa mirrors of 
$X$ and a cubic surface. 
The mirror of $X$ is a Landau-Ginzburg model 
defined by the function $W_y$ on the torus 
$(\C^\times)^2$ with coordinates $x_1,x_2$ 
\begin{align*} 
W_{y}(x_1,x_2) & = a_7 x_1^{-1} x_2^2 + x_2 + 
x_1 + a_8 x_1^2 x_2^{-1} \\
& + a_3 x_1 x_2^{-1} + a_4 x_2^{-1}  
+ a_9 x_1^{-1} x_2^{-1} 
+ a_5 x_1^{-1}  + a_6 x_1^{-1} x_2  
\end{align*} 
where the coefficients 
$a_3,\dots,a_9$ are determined by the relation 
\[
y_j = \prod_{i=3}^9 a_i^{\lan D_i, \gamma_j\ran}, \quad 
1\le j \le 7. 
\]
On the other hand, the mirror of a (generic) 
cubic surface $Y$ is \cite{Prz}
\[
V_u(x_1,x_2) = u \frac{(1+ x_1 + x_2)^3}{x_1x_2}.
\]
Under the specialization 
$y_1 = \cdots = y_6 = 1/3$ and $y_7 = 3^{12} u^3$, 
we have 
\[
V_u(x_1,x_2) = W_y(3 u x_1, 3u x_2) + 6 u.  
\] 
\end{remark} 

% \bibliographystyle{plain}
% \bibliography{../Bib/bib}
\def\cprime{$'$} 
\def\cprime{$'$} 
\def\cprime{$'$} 
\def\cprime{$'$}
\def\cprime{$'$} 
\def\polhk#1{\setbox0=\hbox{#1}{\ooalign{\hidewidth
\lower1.5ex\hbox{`}\hidewidth\crcr\unhbox0}}} 
\def\cprime{$'$}
\def\cprime{$'$}

\end{document}

%% file: box.pdf_t
\begin{picture}(0,0)%
\includegraphics{box.pdf}%
\end{picture}%
\setlength{\unitlength}{3947sp}%
\begingroup\makeatletter\ifx\SetFigFont\undefined%
\gdef\SetFigFont#1#2#3#4#5{%
  \reset@font\fontsize{#1}{#2pt}%
  \fontfamily{#3}\fontseries{#4}\fontshape{#5}%
  \selectfont}%
\fi\endgroup%
\begin{picture}(4824,4824)(3589,-5173)
\put(4501,-4186){\makebox(0,0)[lb]{\smash{{\SetFigFont{12}{14.4}{\rmdefault}{\mddefault}{\updefault}{\color[rgb]{0,0,0}$b_4$}%
}}}}
\put(5776,-1486){\makebox(0,0)[lb]{\smash{{\SetFigFont{12}{14.4}{\rmdefault}{\mddefault}{\updefault}{\color[rgb]{0,0,0}$b_5$}%
}}}}
\put(4276,-2686){\makebox(0,0)[lb]{\smash{{\SetFigFont{12}{14.4}{\rmdefault}{\mddefault}{\updefault}{\color[rgb]{0,0,0}$b_8$}%
}}}}
\put(7201,-2686){\makebox(0,0)[lb]{\smash{{\SetFigFont{12}{14.4}{\rmdefault}{\mddefault}{\updefault}{\color[rgb]{0,0,0}$b_6$}%
}}}}
\put(5776,-4186){\makebox(0,0)[lb]{\smash{{\SetFigFont{12}{14.4}{\rmdefault}{\mddefault}{\updefault}{\color[rgb]{0,0,0}$b_7$}%
}}}}
\put(7201,-1486){\makebox(0,0)[lb]{\smash{{\SetFigFont{12}{14.4}{\rmdefault}{\mddefault}{\updefault}{\color[rgb]{0,0,0}$b_2$}%
}}}}
\put(4501,-1486){\makebox(0,0)[lb]{\smash{{\SetFigFont{12}{14.4}{\rmdefault}{\mddefault}{\updefault}{\color[rgb]{0,0,0}$b_1$}%
}}}}
\put(7276,-4186){\makebox(0,0)[lb]{\smash{{\SetFigFont{12}{14.4}{\rmdefault}{\mddefault}{\updefault}{\color[rgb]{0,0,0}$b_3$}%
}}}}
\end{picture}%

%% file: counter_example.pdf_t
\begin{picture}(0,0)%
\includegraphics{counter_example.pdf}%
\end{picture}%
\setlength{\unitlength}{3947sp}%
\begingroup\makeatletter\ifx\SetFigFont\undefined%
\gdef\SetFigFont#1#2#3#4#5{%
  \reset@font\fontsize{#1}{#2pt}%
  \fontfamily{#3}\fontseries{#4}\fontshape{#5}%
  \selectfont}%
\fi\endgroup%
\begin{picture}(4987,4824)(5836,-6373)
\put(6901,-4186){\makebox(0,0)[lb]{\smash{{\SetFigFont{12}{14.4}{\rmdefault}{\mddefault}{\updefault}{\color[rgb]{0,0,0}$b_8$}%
}}}}
\put(5851,-2611){\makebox(0,0)[lb]{\smash{{\SetFigFont{12}{14.4}{\rmdefault}{\mddefault}{\updefault}{\color[rgb]{0,0,0}$b_1$}%
}}}}
\put(8476,-5311){\makebox(0,0)[lb]{\smash{{\SetFigFont{12}{14.4}{\rmdefault}{\mddefault}{\updefault}{\color[rgb]{0,0,0}$b_3$}%
}}}}
\put(10801,-2611){\makebox(0,0)[lb]{\smash{{\SetFigFont{12}{14.4}{\rmdefault}{\mddefault}{\updefault}{\color[rgb]{0,0,0}$b_2$}%
}}}}
\put(7126,-2611){\makebox(0,0)[lb]{\smash{{\SetFigFont{12}{14.4}{\rmdefault}{\mddefault}{\updefault}{\color[rgb]{0,0,0}$b_4$}%
}}}}
\put(8476,-2611){\makebox(0,0)[lb]{\smash{{\SetFigFont{12}{14.4}{\rmdefault}{\mddefault}{\updefault}{\color[rgb]{0,0,0}$b_5$}%
}}}}
\put(9526,-2611){\makebox(0,0)[lb]{\smash{{\SetFigFont{12}{14.4}{\rmdefault}{\mddefault}{\updefault}{\color[rgb]{0,0,0}$b_6$}%
}}}}
\put(9676,-4186){\makebox(0,0)[lb]{\smash{{\SetFigFont{12}{14.4}{\rmdefault}{\mddefault}{\updefault}{\color[rgb]{0,0,0}$b_7$}%
}}}}
\end{picture}%

%% file: egD.pdf_t
\begin{picture}(0,0)%
\includegraphics{egD.pdf}%
\end{picture}%
\setlength{\unitlength}{3947sp}%
\begingroup\makeatletter\ifx\SetFigFont\undefined%
\gdef\SetFigFont#1#2#3#4#5{%
  \reset@font\fontsize{#1}{#2pt}%
  \fontfamily{#3}\fontseries{#4}\fontshape{#5}%
  \selectfont}%
\fi\endgroup%
\begin{picture}(4834,5070)(5989,-6373)
\put(8476,-2611){\makebox(0,0)[lb]{\smash{{\SetFigFont{12}{14.4}{\rmdefault}{\mddefault}{\updefault}{\color[rgb]{0,0,0}$b_1$}%
}}}}
\put(6901,-5386){\makebox(0,0)[lb]{\smash{{\SetFigFont{12}{14.4}{\rmdefault}{\mddefault}{\updefault}{\color[rgb]{0,0,0}$b_9$}%
}}}}
\put(10726,-5386){\makebox(0,0)[lb]{\smash{{\SetFigFont{12}{14.4}{\rmdefault}{\mddefault}{\updefault}{\color[rgb]{0,0,0}$b_8$}%
}}}}
\put(9526,-5461){\makebox(0,0)[lb]{\smash{{\SetFigFont{12}{14.4}{\rmdefault}{\mddefault}{\updefault}{\color[rgb]{0,0,0}$b_3$}%
}}}}
\put(9676,-3886){\makebox(0,0)[lb]{\smash{{\SetFigFont{12}{14.4}{\rmdefault}{\mddefault}{\updefault}{\color[rgb]{0,0,0}$b_2$}%
}}}}
\put(6601,-3886){\makebox(0,0)[lb]{\smash{{\SetFigFont{12}{14.4}{\rmdefault}{\mddefault}{\updefault}{\color[rgb]{0,0,0}$b_5$}%
}}}}
\put(6601,-2686){\makebox(0,0)[lb]{\smash{{\SetFigFont{12}{14.4}{\rmdefault}{\mddefault}{\updefault}{\color[rgb]{0,0,0}$b_6$}%
}}}}
\put(6676,-1486){\makebox(0,0)[lb]{\smash{{\SetFigFont{12}{14.4}{\rmdefault}{\mddefault}{\updefault}{\color[rgb]{0,0,0}$b_7$}%
}}}}
\put(8476,-5461){\makebox(0,0)[lb]{\smash{{\SetFigFont{12}{14.4}{\rmdefault}{\mddefault}{\updefault}{\color[rgb]{0,0,0}$b_4$}%
}}}}
\end{picture}%

%% file: Batyrev-Seidel-selecta.bbl
\begin{thebibliography}{}
\providecommand\bibmarginpar{\leavevmode\marginpar}
\def\urlstyle#1{{\tt #1}}

\bibitem{Au-To04}
\textbf{M Audin}, \emph{Torus actions on symplectic manifolds}, volume~93 of
  \emph{Progress in Mathematics}, revised edition, Birkh\"auser Verlag, Basel
  (2004)

\bibitem{Ba-Qu93}
\textbf{V\,V Batyrev}, \emph{Quantum cohomology rings of toric manifolds},
  Ast\'erisque  (1993) 9--34Journ{\'e}es de G{\'e}om{\'e}trie Alg{\'e}brique
  d'Orsay (Orsay, 1992)

\bibitem{CCIT} 
\textbf{T Coates}, \textbf{A Corti}, \textbf{H Iritani}, \textbf{H-H Tseng} 
\emph{Computing genus zero twisted Gromov-Witten invariants}, 
Duke Math J.\ Vol 147, No.3 (2009) 377--438. 

\bibitem{CoKa-Mi99}
\textbf{D\,A Cox}, \textbf{S Katz}, \emph{Mirror symmetry and algebraic
  geometry}, volume~68 of \emph{Mathematical Surveys and Monographs}, American
  Mathematical Society, Providence, RI (1999)

\bibitem{CoLi-To10}
\textbf{D\,A Cox}, \textbf{J Little}, \textbf{H Schenck}, \emph{Toric
  Varieties}, American Mathematical Society (2010)

\bibitem{Fr-59} 
\textbf{T Frankel}, 
\emph{Fixed points and torsion on K\"{a}hler manifolds}, 
Ann.\ of Math.\ (2) 70 (1959) 1--8. 

\bibitem{FuOh-La10A}
\textbf{K Fukaya}, \textbf{Y-G Oh}, \textbf{H Ohta}, \textbf{K Ono},
  \href{http://arxiv.org/abs/1009.1648v1} {\emph{Lagrangian Floer theory and
  mirror symmetry on compact toric manifolds}}, ArXiv Mathematics e-prints
  1009.1648v1 (2010)

\bibitem{Gi-Eq96}
\textbf{A\,B Givental}, \emph{Equivariant {G}romov-{W}itten invariants},
  Internat. Math. Res. Notices  (1996) 613--663

\bibitem{Gi-A-98}
\textbf{A Givental}, \emph{A mirror theorem for toric complete intersections},
  from: ``Topological field theory, primitive forms and related topics
  ({K}yoto, 1996)'', Progr. Math. 160, Birkh\"auser Boston, Boston, MA (1998)
  141--175

\bibitem{Go-Qu06}
\textbf{E Gonzalez}, \href{http://dx.doi.org/10.1090/S0002-9947-06-04038-4}
  {\emph{Quantum cohomology and {$S^1$}-actions with isolated fixed points}},
  Trans. Amer. Math. Soc. 358 (2006) 2927--2948 (electronic)

\bibitem{Gu-Mo94}
\textbf{V Guillemin}, \emph{Moment Maps and Combinatorial Invariants of
  {H}amiltonian {$T^n$}-Spaces}, volume 122 of \emph{Progress in Mathematics},
  Birkh\"{a}user, Boston (1994)


\bibitem{LaMcPo-To99}
\textbf{F Lalonde}, \textbf{D McDuff}, \textbf{L Polterovich},
  \href{http://dx.doi.org/10.1007/s002220050289} {\emph{Topological rigidity of
  {H}amiltonian loops and quantum homology}}, Invent. Math. 135 (1999) 369--385

\bibitem{LiTi-Co99}
\textbf{J Li}, \textbf{G Tian}, \emph{Comparison of algebraic and symplectic
  {G}romov-{W}itten invariants}, Asian J. Math. 3 (1999) 689--728

\bibitem{Mc-Qu00}
\textbf{D McDuff}, \href{http://dx.doi.org/10.1142/S0129167X00000337}
  {\emph{Quantum homology of fibrations over {$S^2$}}}, Internat. J. Math. 11
  (2000) 665--721

\bibitem{Mc-Sy06}
\textbf{D McDuff}, \emph{Symplectomorphism groups and quantum cohomology},
  from: ``The unity of mathematics'', Progr. Math. 244, Birkh\"auser Boston,
  Boston, MA (2006)  457--471

\bibitem{McSa-J04}
\textbf{D McDuff}, \textbf{D Salamon}, 
\emph{$J$-holomorphic curves and symplectic topology}, 
American Mathematical Society Colloquium Publications, 
52. American Mathematical Society, Providence, RI, 2004. 

\bibitem{McTo-To06}
\textbf{D McDuff}, \textbf{S Tolman}, \emph{Topological properties of
  {H}amiltonian circle actions}, IMRP Int. Math. Res. Pap.  (2006) 72826, 1--77

\bibitem{Prz} 
\textbf{V Przyjalkowski}, 
\emph{Hori-Vafa mirror for complete intersections in weighted projective space 
and weak Landau-Ginzburg model}, arXiv:1003.5200. 

\bibitem{Se-pi97}
\textbf{P Seidel}, \emph{{$\pi\sb 1$} of symplectic automorphism groups and
  invertibles in quantum homology rings}, Geom. Funct. Anal. 7 (1997)
  1046--1095

\end{thebibliography}
